\documentclass[12pt, reqno]{amsart}
\usepackage[utf8]{inputenc}
\usepackage[english]{babel}
\usepackage{todonotes}
\usepackage{comment}
\usepackage[normalem]{ulem}

\usepackage{blkarray}
\usepackage{fullpage}              
\usepackage{tikz}                       
\usetikzlibrary{plotmarks}
\usetikzlibrary{shapes}
\usetikzlibrary{arrows}
\usepackage{graphicx}
\usepackage{caption}
\usepackage{subcaption}
\usepackage[pdftex]{hyperref}
\hypersetup{
    colorlinks=true,
    linkcolor={MidnightBlue},
    citecolor={magenta},
    urlcolor={MidnightBlue}
}
\usepackage[nameinlink,capitalise]{cleveref}
\usepackage{soul}
\usepackage{algorithm2e}
\SetKwComment{Comment}{/* }{ */}
\usepackage[symbol]{footmisc}

\numberwithin{equation}{section}
\theoremstyle{plain}
\newtheorem{theorem}{Theorem}[section]

\newtheorem{corollary}[theorem]{Corollary}
\newtheorem{lemma}[theorem]{Lemma}
\newtheorem{proposition}[theorem]{Proposition}

\newtheorem{manualtheoreminner}{}
\newenvironment{manualtheorem}[1]{%
  \renewcommand\themanualtheoreminner{#1}%
  \manualtheoreminner
}{\endmanualtheoreminner}

\theoremstyle{definition}
\newtheorem{remark}[theorem]{Remark}
\newtheorem{examplex}[theorem]{Example}
\newenvironment{example}
  {\pushQED{\qed}\examplex}
  {\popQED\endexamplex}

\newtheorem{definition}[theorem]{Definition}

\newcommand{\Cc}{{\mathbb C}}
\newcommand{\Nn}{{\mathbb N}}
\newcommand{\Qq}{{\mathbb Q}}
\newcommand{\Rr}{{\mathbb R}}
\newcommand{\Zz}{{\mathbb Z}}
\newcommand{\Pp}{{\mathbb P}}

\newcommand{\VV}{\mathcal{V}}

\newcommand{\cs}{\mathcal{S}}
\newcommand{\cn}{\mathcal{N}}
\newcommand{\ct}{\mathcal{T}}

\newcommand{\ci}{\mathcal{I}}

\DeclareMathOperator{\sign}{sign}

\DeclareMathOperator{\image}{image}
\DeclareMathOperator{\Pf}{Pf}	
\DeclareMathOperator{\Skew}{Skew}
\newcommand{\zmodtwo}{\Zz/2\Zz}
\newcommand{\isom}{\cong}

\definecolor{benpurple}{RGB}{180, 0, 240}
\definecolor{joered}{RGB}{240,0,0}
\definecolor{samcyan}{RGB}{0,240,250}
\definecolor{MidnightBlue}{RGB}{25,25,112}

\title{The Pfaffian Structure of CFN Phylogenetic Networks}

\author{Joseph Cummings}
\address{University of Notre Dame, Indiana, USA}
\email{josephcummings03@gmail.com}
\author{Elizabeth Gross}
\address{University of Hawai`i at M$\bar{\text{a}}$noa, Hawai`i, USA}
\email{egross@hawaii.edu}

\author{Benjamin Hollering}
\address{Technische Universit\"at M\"unchen, DE}
\email{benhollering@gmail.com}

\author{Samuel Martin}
\address{Earlham Institute, Norwich, UK}
\address{European Bioinformatics Institute, EMBL-EBI, Hinxton, UK}
\email{samuel.martin@ebi.ac.uk}

\author{Ikenna Nometa}
\address{University of Hawai`i at M$\bar{\text{a}}$noa, Hawai`i, USA}
\email{inometa@hawaii.edu}

\keywords{phylogenetic network, Markov model, group-based model, 
phylogenetic invariant, algebraic variety, Pfaffian}
\subjclass{92B10, 62R01, 13P25}

\begin{document}
\maketitle

\begin{abstract}
Algebraic techniques in phylogenetics have historically been successful at proving identifiability results and have also led to novel reconstruction algorithms. In this paper, we study the ideal of phylogenetic invariants of the Cavender-Farris-Neyman (CFN) model on a phylogenetic network with the goal of providing a description of the invariants which is useful for network inference. It was previously shown that to characterize the invariants of any level-1 network, it suffices to understand all sunlet networks, which are those consisting of a single cycle with a leaf adjacent to each cycle vertex. We show that the parameterization of an affine open patch of the CFN sunlet model, which intersects the probability simplex, factors through the space of skew-symmetric matrices via Pfaffians.  We then show that this affine patch is isomorphic to a determinantal variety and give an explicit Gr{\"o}bner basis for the associated ideal, which involves only $\binom{n}{2}$ coordinates rather than $2^{n}$. Lastly, we show that sunlet networks with at least 6 leaves are identifiable using only these polynomials and run extensive simulations, which show that these polynomials can be used to accurately infer the correct network from DNA sequence data. 
\end{abstract}

\section{Introduction}\label{sec-1::intro}
Phylogenetic networks are directed acyclic graphs that are used to represent the evolutionary history of a set of species where reticulation events, such as hybridization or horizontal gene transfer, may have occurred. A phylogenetic network is typically equipped with a parametric statistical model of evolution, which enables phylogenetic analyses and phylogenetic network inference from genome data. One recent approach to phylogenetic network inference is rooted in computational algebraic geometry.  For example, algebraic methods have recently been used to establish identifiability for several classes of phylogenetic network models \cite{allman2019species,banos2019identifying, gross2023dimensions, gross2018distinguishing, gross2021distinguishing, hollering2021identifiability} and for phylogenetic network inference \cite{allman2019nanuq, barton2022statistical,  martin2023algebraic, wu2022ultrafast}. In particular, \cite{barton2022statistical} and \cite{martin2023algebraic} investigate using algebraic techniques to infer phylogenetic networks from molecular sequence data. Most algebraic results on identifiability or inference utilize the \emph{phylogenetic invariants} of the model, which are polynomials that vanish on every distribution in the model; thus, one common problem for any statistical model associated with a phylogenetic network is to characterize the set of invariants of the model.

While ideals of phylogenetic invariants are useful for different statistical inference problems, very little is known about these ideals for \emph{network-based Markov models}, or equivalently, \emph{displayed tree models}. In \cite{casanellas2021rank}, a partial description of the invariants for equivariant models is given, but these invariants do not suffice to reconstruct a network entirely. One natural place to begin is the study of \emph{group-based} Markov models since the discrete Fourier transform can be used to greatly simplify the parameterization \cite{evans1993invariants, hendy1996complete}, which makes it much more amenable to study from an algebraic perspective. This is further evidenced by the fact that invariants of group-based models for trees are completely characterized \cite{sturmfels2005toric}, but for many other equivariant models, they are still unknown \cite{allman2008gmm}.  The study of the invariants for group-based models on networks was started in  \cite{gross2018distinguishing} with a focus on identifiability. More recently, in \cite{cummings2024invariants}, the authors show that the invariants of any group-based model on a family of phylogenetic networks called \emph{level-1} can be completely determined from the invariants on phylogenetic trees and \emph{sunlet networks} under the same model. Sunlet networks are a type of phylogenetic network that consist of a single cycle, with one leaf adjacent to each vertex in the cycle. Formal definitions of level-1 phylogenetic networks and sunlet networks are given in Section {\ref{sec2::Prelim}}. This result mirrors the well-known graph decomposition of level-1 phylogenetic networks into trees and sunlets.

Typically, methods in molecular phylogenetics use DNA sequence data, so that the states in evolutionary models are the four nucleotides A, C, G, and T. One way to simplify the analysis of DNA sequence data is through RY-coding, where these nucleotides are recoded to purines (R) and pyrimidines (Y). The Cavender-Farris-Neyman (CFN) model is a two-state Markov model of evolution, where the two states are purines and pyrimidines and each transition matrix is assumed to be symmetric. It is also known variously as the 2-state symmetric model, the binary symmetric model, and the Jukes-Cantor binary model. In the context of group-based models, the CFN model is the \emph {general group-based model} for the group $\mathbb{Z}/2\mathbb{Z}$ {\cite{algstat}}. We formally define the CFN model in Section {\ref{sec2::Prelim}}. In {\cite{cummings2024invariants}}, the authors identify all quadratic invariants in the sunlet network ideal for the CFN model.

In this paper, we complete the study of the invariants of the CFN model on a sunlet network. In particular, we show that the parameterization of the CFN sunlet network model factors through the set of skew-symmetric matrices via Pfaffians, allowing us to work with the model in a new coordinate system. While this is not directly related to the results in \cite{ColbyPfaff}, it does provide another surprising link between Pfaffians and phylogenetics.  The following theorem summarizes our main results. 

\begin{manualtheorem}{Theorem}
   Let $\VV_{\cs_n} \subseteq \Pp^{2^{n-1} - 1}$ be the \emph{projective} phylogenetic variety associated to the CFN model on a $n$-leaf sunlet network for $n \geq 4$. 
   Denote the affine open patch of $\VV_{\cs_n}$ where the sum of the coordinates is 1 by $U_n$.
    Then:
    \begin{enumerate}
        \item [(a)] $U_n$ \emph{embeds} into the space of $n\times n$ skew-symmetric matrices via taking sub-Pfaffians.
        \item [(b)] Let $\Omega$ be a $n\times n$ skew-symmetric matrix with variable entries $x_{ij}$. In these coordinates, $U_n$ is cut out by the following minors of $\Omega$:
        \begin{enumerate}
            \item [(i)] $\det(\Omega_{\{i,j\},\{k,\ell\}})$ where $1 < i < j < k < \ell \leq n$ and 
            \item [(ii)] $\det(\Omega_{\{1,i_2,i_3\}, \{j_1,j_2,j_3\}})$ where $1 < i_2 < j_1 < j_2 < j_3 < i_3 \leq n$,
        \end{enumerate}
        where $\Omega_{S,T}$ is the submatrix of $\Omega$ whose rows and columns are indexed by the elements of $S$ and $T$, respectively.
    \end{enumerate} 
\end{manualtheorem}

The subset $U_n$ of $\VV_{\cs_n}$ where the sum of probability coordinates is $1$ is an affine variety that contains the statistical model, and so we call it the \emph{statistically relevant affine open patch}. For phylogenetic trees, these varieties have been well-studied. For example, in \cite{casanellas2007geometry}, the authors study the analogous variety for unrooted phylogenetic trees under the Kimura 3-parameter model, which they term the Kimura variety for a tree $T$. Here, we also show that the invariants of the original model in projective space can be obtained from the minors in part (b) of the above theorem, along with the relations between the Pfaffians of a skew-symmetric matrix by homogenization and saturation. These results essentially give a complete description of the invariants of the CFN model. Lastly, we use this new description of the CFN model to show that sunlet networks with at least 6 leaves are generically identifiable from data observed at the leaves and provide simulations that showcase the effectiveness of using these determinantal invariants for phylogenetic network inference.

The remainder of this paper is organized as follows. In \cref{sec2::Prelim}, we review background material on phylogenetic networks, network-based Markov models, and Pfaffians. In \cref{sec:3}, we show that the parameterization of the sunlet network model factors through the parameterization of the Pfaffians of a skew-symmetric matrix and utilize this to give a complete description of the invariants of the statistically relevant affine open patch of the model in a new coordinate system. In \cref{sec:inference}, we prove that sunlet networks with at least 6 leaves are generically identifiable under the CFN model and then show that these invariants can be used to accurately infer the correct network from simulated DNA sequence data.

\section{Preliminaries}\label{sec2::Prelim}
This section provides a brief overview of the Cavender-Farris-Neyman (CFN) model on phylogenetic trees and networks. We then provide some background on the structure of the vanishing ideal of the CFN model associated to a phylogenetic tree, which was characterized in \cite{sturmfels2005toric}, and review the multigrading described in \cite{cummings2024invariants}. 

\subsection{Phylogenetic Networks}
In this subsection, we provide some brief background on phylogenetic networks. Our notation and terminology are adapted from \cite{cummings2024invariants, gross2018distinguishing}. For more information on the combinatorics of phylogenetic networks, we refer the reader to \cite{gross2020phylogenetic, steel2016phylogeny}. 

A \emph{binary phylogenetic tree} $\ct$ on $n$ leaves is a tree (i.e. a directed graph with no cycles) with a distinguished vertex of degree 2, called the \emph{root} and denoted $\rho$, and such that every other vertex has either degree 3 or degree 1. The vertices of degree 1 are called \emph{leaves} and there are exactly $n$ of them. In applications, the leaves of $\ct$ are labelled by the names of taxa, however we will always label leaves by the set of integers $[n]=\{1, \ldots, n\}$. We denote by $E(\ct)$ the set of edges of $\ct$, and by $V(\ct)$ the set of vertices of $\ct$. 

A binary \emph{phylogenetic network} on $n$ leaves is a rooted, directed, acyclic graph, where the root vertex is a distinguished vertex of out-degree two, and all other vertices have either in-degree one and out-degree two (these are called \emph{tree vertices}), in-degree two and out-degree one (these are called \emph{reticulation vertices}), or in-degree one and out-degree zero (these are called \emph{leaf vertices}). The set of leaf vertices is identified with the set of integers $[n]$. Edges directed into a reticulation vertex are called \emph{reticulation edges}, and all other edges are called \emph{tree edges}. In this paper, we will focus on the \emph{Cavender-Farris-Neyman} model \cite{cavender1978taxonomy, farris1973probability, neyman1971statistical}, which is \emph{group-based} (see Definition {\ref{def:groupbased}} below) and thus \emph{time-reversible} (in the sense of Markov chains, see e.g. {\cite{norris1997markov}}). This implies that it is impossible to identify the location of the root under these models, so we restrict our attention to the underlying \emph{semi-directed} network of the phylogenetic network. The underlying semi-directed network is a mixed graph obtained by suppressing the root and undirecting all tree edges in the network. The reticulation edges remain directed into the reticulation vertex. Since the reticulation edges are implicitly directed into the reticulation vertex, we typically omit the arrows when drawing semi-directed networks. This is illustrated in Figure \ref{fig:RootedToSemidirected}.

\begin{figure}
    \centering
    \begin{subfigure}[b]{0.3\linewidth}
        \centering
        \begin{tikzpicture}[scale = .3, thick]
        \draw [fill] (4,8) circle [radius = 0.1]; 
        \draw [fill] (2,6) circle [radius = 0.1];
        \draw [fill] (0,4) circle [radius = 0.1];
        \draw [fill] (4,4) circle [radius = 0.1];
        \draw [fill] (2,2) circle [radius = 0.1];
        \draw [fill] (0,0) circle [radius = 0.1];
        \draw [fill] (-4,0) circle [radius = 0.1];
        \draw [fill] (8,0) circle [radius = 0.1];
        \draw [fill] (12,0) circle [radius = 0.1];
        
        \draw (4,8)--(2,6);
        \draw (2,6)--(0,4);
        \draw (2,6)--(4,4);
        \draw [dashed] (0,4)--(2,2);
        \draw [dashed] (4,4)--(2,2);
        
        \draw (2,2)--(0,0);
        \draw (0,4)--(-4,0);
        \draw (4,4)--(8,0);
        \draw (4,8)--(12,0);
        
        \draw (-4,0) node[below]{$2$};
        \draw (0,0) node[below]{$1$};
        \draw (8,0) node[below]{$4$};
        \draw (12,0) node[below]{$3$};
    \end{tikzpicture}%
    \end{subfigure}
    \begin{subfigure}[b]{0.3\linewidth}
        \centering
        \begin{tikzpicture}[scale = .5, thick]
        \draw [dashed] (2,2)--(4,2);
        \draw (4,2)--(4,4);
        \draw (4,4)--(2,4);
        \draw [dashed] (2,4)--(2,2);
        
        \draw (2,2)--(1,1);
        \draw (4,2)--(5,1);
        \draw (4,4)--(5,5);
        \draw (2,4)--(1,5);
        
        \draw (1,1) node[below]{$1$};
        \draw (5,1) node[below]{$2$};
        \draw (5,5) node[above]{$3$};
        \draw (1,5) node[above]{$4$};
        \end{tikzpicture}
    \end{subfigure}
    \begin{subfigure}[b]{0.3\linewidth}
        \centering
        \begin{tikzpicture}[scale = .5, thick]
        \draw [dashed] (2,2)--(4,2);
        \draw (4,2)--(4,4);
        \draw (4,4)--(2,4);
        \draw [dashed] (2,4)--(2,2);
        
        \draw (2,2)--(1,1);
        \draw (4,2)--(5,1);
        \draw (4,4)--(5,5);
        \draw (2,4)--(1,5);
        
        \draw (5,1)--(7,1);
        \draw (7,1)--(7,-1);
        \draw [dashed] (7,-1)--(5,-1);
        \draw [dashed] (5,-1)--(5,1);
        
        \draw (7,1)--(8,2);
        \draw (7,-1)--(8, -2);
        \draw (5,-1)--(4, -2);
        
        \draw (1,1) node[below]{$1$};
        \draw (4.5,1.5) node[left, below]{};
        \draw (5,5) node[above]{$3$};
        \draw (1,5) node[above]{$2$};
        \draw (8,2) node[above]{$6$};
        \draw (8,-2) node[below]{$5$};
        \draw (4,-2) node[below]{$4$};
        \end{tikzpicture}
    \end{subfigure}
    \caption{A four-leaf, level-1 phylogenetic network pictured on the left with all edges directed away from the root. Reticulation edges are drawn with dashed lines. In the center is the associated semi-directed network obtained by suppressing the root and undirecting all tree edges. The edges are implicitly assumed to be directed into the vertex adjacent to the leaf $1$, which is the reticulation vertex. On the right, a semi-directed level-1 network with $2$ biconnected components is pictured.}
    \label{fig:RootedToSemidirected}
\end{figure}

The definition of phylogenetic networks above is very general and allows for an extremely wide array of structures, making it difficult to analyze or study phylogenetic networks in general. This means that subclasses of networks are often studied instead. In this paper, we restrict to the subclass of \emph{level-1} phylogenetic networks, which are phylogenetic networks with at most 1 cycle in each 2-connected component. This class of networks is particularly nice since they can be formed by gluing trees and \emph{sunlet networks} together along leaves.

\begin{definition}
A \emph{$n$-sunlet network} is a semi-directed phylogenetic network with one reticulation vertex and whose underlying graph is obtained by adding a pendant edge to every vertex of a $n$-cycle. We denote with $\cs_n$ the $n$-sunlet network with reticulation vertex adjacent to the leaf 1 and the other leaves labeled circularly from 1. 
\end{definition}

Throughout this work we will denote the leaf edges of $\cs_n$, leading to leaves $1,\ldots, n$ by $e_1$,\ldots, $e_n$ respectively. The internal cycle edges of $\cs_n$ will be denoted by $e_{n+1}, \ldots, e_{2n}$, where for $i=1,\ldots n$, edge $e_{n+i}$ is the cycle edge between edges $e_i$ and $e_{i+1}$. In particular, edges $e_{n+1}$ and $e_{2n}$ are the two reticulation edges of $\cs_n$. See Figure {\ref{fig:4SunletAndTrees}}(A) for an example. The class of sunlet networks was first introduced in \cite{gross2018distinguishing} and was further studied in \cite{cummings2024invariants}. Here, we are interested in determining the \emph{invariants} of phylogenetic network models (defined below). Throughout this paper, we will focus primarily on sunlet networks since the invariants of any level-1 network under the CFN model (or indeed, any group-based model) can be determined if the invariants of all $n$-leaf sunlet networks and all phylogenetic trees are known \cite{gross2018distinguishing, cummings2024invariants}.

\subsection{Phylogenetic Markov models and the Cavender-Farris-Neyman Model}
In this subsection, we briefly describe phylogenetic Markov models with a particular focus on the Cavender-Farris-Neyman (CFN) model which is the focus of this paper.  We begin with phylogenetic Markov models on trees since they will be the main ingredient for constructing such models on phylogenetic networks. 

A $\kappa$-state Markov model on a $n$-leaf phylogenetic tree $\ct$ yields a distribution on all possible states that can be jointly observed at the leaves of $\ct$. 
These states are also called marginal characters and are often interpreted as being the possible columns representing a single site in aligned sequence data. For example, when $\kappa = 4$ and the state space is the nucleic acids $\{\rm{A,G,C,T}\}$, each joint state is a column that can appear in an alignment of DNA sequence data for $n$ taxa.

The distribution of joint states is produced by associating a $\kappa$-state random variable $X_v \in [\kappa] = \{1, \ldots, \kappa\}$ to each vertex $v\in V(\ct)$, and a $\kappa \times \kappa$ transition matrix $M^e$ to each directed edge $e=(u, v) \in E(\ct)$ such that $M_{i,j}^e = P(X_v = j | X_u = i)$. Lastly, associate a distribution $\pi$ of states to the root $\rho$ of $\ct$. 
Then the probability of observing a joint state $(x_1, \ldots, x_n) \in [\kappa]^n$ at the leaves of $\ct$ is given by marginalizing over the internal nodes of $\ct$. We can write down an expression for this in the following manner. Let $X(x_1,\ldots, x_n) \subset [\kappa]^{|V(\ct)|}$ be the set of all possible assignments of states at all vertices in $\ct$ that have states $x_1,\dots, x_n$ assigned to leaves $1,\ldots,n$ respectively. Let ${\rm Int}(\ct)$ denote the set of interior vertices of $\ct$, and for $x\in X(x_1,\ldots, x_n)$ and $u\in V(\ct)$, let $x_u$ denote the state assigned to vertex $u$. Then we have
\[
p_{x_1, x_2, \ldots, x_n} := P(X_1 = x_1, \ldots, X_n = x_n) ~= 
\sum_{x \in X(x_1,\ldots, x_n)}\pi_{x_\rho}\prod_{(u,v) \in E(\ct)}M_{x_u, x_v}^{(u,v)}, 
\]
where $X_1,\ldots,X_n$ are the random variables associated to the leaves of $\ct$. When the meaning is clear, we will drop the commas in $p_{x_1, x_2, \ldots, x_n}$ and simply write $p_{x_1 x_2 \ldots x_n}$. Observe that the joint distribution $p = (p_{x_1 x_2 \ldots x_n} ~|~ (x_1, \ldots, x_n) \in [\kappa]^n)$ of $(X_1, \ldots X_n)$ is given by polynomials in the entries of $\pi$ and $M^e$. In other words, the model can be thought of as the image of a  polynomial map
\begin{align*}
    \Psi_\ct ~:~ \Theta_\ct &\to \Delta_{\kappa^n-1} = \{p \in \Rr^{\kappa^n} ~|~ \sum_{(x_1, \ldots, x_n) \in [\kappa]^n} p_{x_1 x_2 \ldots x_n} = 1 ~\mathrm{and}~ p_{x_1 x_2 \ldots x_n } \geq 0 \}\\ 
             (\pi, M^e) &\mapsto 
             p = \Big(p_{x_1x_2\ldots x_n} =\sum_{x \in X(x_1,\ldots, x_n)}\pi_{x_\rho}\prod_{(u,v) \in E(\ct)}M_{x_u, x_v}^{(u,v)}\ |\ (x_1, \ldots, x_n) \in [\kappa]^n\Big),
\end{align*}
where $\Theta_\mathcal{T}$ is the \emph{stochastic parameter space} of the model, which consists of all possible tuples $(\pi, M^e ~|~ e \in E(\ct))$ of root distributions and transition matrices. The phylogenetic model $M_\ct = \image(\Psi_\ct)$ is the set of all possible distributions of joint states at the leaves of $\ct$ (often called leaf-patterns), and is the observable part of the model. Since $M_\ct$ is the image of a polynomial map, tools from commutative algebra and algebraic geometry can be used to study it, which is the main idea behind \emph{algebraic statistics} \cite{algstat}. In particular, algebraic tools can be used to determine invariants of the model, prove identifiability results, and can even yield algorithms for model selection, as we will see throughout the remainder of this paper. 

We may now formally define the CFN model.
\begin{definition}\label{def:cfn}
The \emph{Cavender-Farris-Neyman} (CFN) model on a phylogenetic tree $\ct$ is the 2-state Markov model for which the state space is $\{R,Y\}$, representing purines and pyrimidines, the distribution at the root vertex of $\ct$ is $\pi = (\frac{1}{2}, \frac{1}{2})$, and every transition matrix $M^e$ for $e\in E(\ct)$ has the form
$$M^e = \begin{pmatrix} \alpha_e & \beta_e \\ \beta_e & \alpha_e \end{pmatrix}$$
with $\alpha_e, \beta_e \in [0,1]$ such that $\alpha_e + \beta_e = 1$.
\end{definition}

Phylogenetic Markov models for trees naturally induce a model on phylogenetic networks in the following way. Let $\cn$ be a phylogenetic network with reticulation vertices $v_1, \dots, v_m$, and let $e_i^0$ and $e_i^1$ be the reticulation edges adjacent to $v_i$. As in the case for trees, associate a transition matrix to each edge of $\cn$ and a distribution of states at the root $\rho$. Independently at random, delete $e_i^0$ with probability $\lambda_i$ and otherwise delete $e_i^1$ and record which edge is deleted with a vector $\sigma \in \{0,1\}^m$ where $\sigma_i = 0$ indicates that edge $e_i^0$ was deleted. Each $\sigma$ corresponds to a different tree $\ct_\sigma$, which corresponds to the different possible phylogenetic trees for the species at the leaves of the network. The parameterization $\Psi_\cn$ is given by
\begin{equation}
\label{eq:networkParam}
    \Psi_\cn = \sum_{\sigma \in \{0,1\}^m}\left(\prod_{i=1}^m \lambda_i^{1 - \sigma_i}(1-\lambda_i)^{\sigma_i} \right)\Psi_{\ct_\sigma}
\end{equation}
where $\Psi_{\ct_\sigma}$ is the parameterization corresponding to the tree $\ct_\sigma$ with transition matrices inherited from the original network $\cn$. 
On first glance, the phylogenetic network model looks like the $2^m$-component mixture model on $\mathbf{T}$, where $\mathbf{T} = (\ct_\sigma \,|\, \sigma\in\{0,1\}^m)$ {\cite{rhodes2012identifiability}}. However, in the $2^m$-component mixture model case, the Markov parameters on the trees are independent, whereas in the phylogenetic network model above, parameters for corresponding edges present in multiple $\ct_\sigma$ (e.g. all non-reticulation edges) are identified. Thus, we can think of the model for the phylogenetic network $\cn$ as a submodel of the $2^m$-component mixture model on $\mathbf{T}$, in the sense that the set of distributions of states at the leaves for the phylogenetic network model is a subset of that for the mixture model.

When no restrictions are place on the root distribution or transition matrices, the model (for either a phylogenetic tree or phylogenetic networks) is called the \emph{general Markov model}. This has received extensive study for phylogenetic trees, but given the large numbers of parameters, it can be difficult to analyze \cite{allman2003invariants, allman2008gmm}. Often, simpler models of DNA evolution are specified by requiring that the transition matrices satisfy additional constraints. One well-studied family of models for phylogenetic trees is the following. 

\begin{definition}\label{def:groupbased}
Let $G$ be a finite abelian group of order $\kappa$ and $\ct$ a rooted binary tree. Then a \emph{group-based model} on $\ct$ is a phylogenetic Markov model on $\ct$ such that for each transition matrix $M^e$, there exists a function $f_e: G \to \Rr$ such that $M_{g,h}^e = f_e(g-h)$. 
\end{definition}

Here, we identify the group $G$ with the state space of the model. Several well-known phylogenetic models are group-based, such as the Jukes-Cantor, Kimura 2-Parameter, and Kimura 3-Parameter models and, of course, the CFN model, which is the main focus of this paper. Group-based models are particularly amenable to study from an algebraic perspective since they allow a linear change of coordinates, which vastly simplifies the parameterization of the model by simultaneously diagonalizing the transition matrices. This change of coordinates is called the discrete Fourier transform and was first applied to phylogenetic models in \cite{evans1993invariants, hendy1996complete}. It is applied to both the model parameters (i.e. the Markov matrices $M^e$ for each edge $e$) and the probability coordinates $p_{x_1x_2\ldots x_n}$.  After applying this change of coordinates, the transformed probability coordinates are typically called the \emph{Fourier coordinates} and are denoted by $q_{g_1 g_2 \ldots g_n}$ for $g_1, g_2,\ldots,g_n \in G$, where we now make the identification of $G$ with the state space of the model, so that $g_i = x_i$ for $i=1,\ldots, n$. A further simplification of the model is made with the observation that if $g_1 + \cdots + g_n \neq 0$, then $q_{g_1\ldots g_n} = 0$ (see e.g. \cite[Corollary~15.3.17]{algstat}). These coordinates are therefore ignored, and the dimension of the space in which the model lives can be reduced by a factor of $|G|$. 

In what follows, we give the parameterization for the CFN model after this change of coordinates and refer the reader to {\cite{steel2016phylogeny, algstat}} for more detailed information. First, we associate the state space $\{R,Y\}$ with the group $\zmodtwo = \{0,1\}$. The Fourier coordinates can be written explicitly in terms of the probability coordinates as follows:
\begin{equation}\label{eqn:fouriercoords}
    q_{g_1\dotsc g_n} = \sum_{(h_1,\dotsc,h_n) \in (\zmodtwo)^n} (-1)^{\sum_{i=1}^n g_i h_i}\ p_{h_1\ldots h_n}.
\end{equation}

In Fourier coordinates, the model associated to a phylogenetic tree $\ct$ is parameterized using the \emph{splits} of $\ct$. Each edge $e \in \ct$ has an associated split $A_e | B_e$, which is the partition of the leaf set of $\ct$ obtained by deleting the edge $e \in E(\ct)$.
To parameterize the CFN model in the Fourier coordinates, we define parameters $a_g^e$ for each $g \in \zmodtwo$ and edge $e$. 

\begin{remark}\label{remark:params}
The parameters $a^e_g$ can be recovered from the transition matrices $M^e$ in Definition \ref{def:cfn}. For the CFN model, these matrices take the form
\[
M^e = \begin{pmatrix} \alpha_e & \beta_e \\ \beta_e & \alpha_e \end{pmatrix},
\]
and the parameters $a_0^e$ and $a_1^e$ are $\alpha_e + \beta_e$ and $\alpha_e - \beta_e$, respectively. These are the eigenvalues of $M^e$ (see \cite[\S 3]{evans1993invariants}).
\end{remark}

As our goal is to use to use tools from algebraic geometry to find phylogenetic invariants, we consider a relaxation of the problem. If we were to keep all the restrictions on the parameters, then we would have that $a_0^e = 1$ and $a_1^e \in [-1,1]$; however, in what follows, we will allow $a_g^e$ to take any complex value; thereby, extending $\Psi_\ct$ to a complex polynomial map. This has the effect of allowing us to open the algebraic geometry toolbox while not affecting the invariant problem at all. We will therefore not rename the map and still call it $\Psi_\ct$. With this in mind, the \emph{parameterization map in Fourier coordinates} is given by 
\[
\Psi_{\ct}:\Cc^{2|E(\ct)|}\longrightarrow \Cc^{2^{n-1}}
\]
where
\[
(\Psi_\ct(a))_{g_1\ldots g_n} = \prod_{e \in E(\ct)} a_{\sum_{i \in A_e}g_i}^e.
\]
Here, for $a\in \Cc^{2|E(\ct)|}$, the coordinates of each factor of $\Cc$ are given by $a_0^e$ or $a_1^e$ for each $e \in E(\ct)$, and we identify $\Cc^{2^{n-1}}$ with the space of Fourier coordinates $q_{g_1\ldots g_n}$ such that $g_1 + \cdots + g_n = 0$. We may also write 
\begin{equation}
\label{eqn:TreeParam}
q_{g_1\ldots g_n} = 
\begin{cases}
\prod_{e \in E(\ct)} a_{\sum_{i \in A_e}g_i}^e & \mbox{ if } \sum_{i \in [n]}g_i = 0 \\ 
0 & \mbox{ otherwise.}
\end{cases}
\end{equation}

\begin{figure}
    \centering
    \begin{subfigure}[b]{0.3\linewidth}
        \centering
        \begin{tikzpicture}[scale = .5, thick]
        \draw [dashed] (2,2)--(4,2);
        \draw (4,2)--(4,4);
        \draw (4,4)--(2,4);
        \draw [dashed] (2,4)--(2,2);
        
        \draw (2,2)--(1,1);
        \draw (4,2)--(5,1);
        \draw (4,4)--(5,5);
        \draw (2,4)--(1,5);
        
        \draw (1,1) node[below]{$1$};
        \draw (5,1) node[below]{$2$};
        \draw (5,5) node[above]{$3$};
        \draw (1,5) node[above]{$4$};
        
        \draw (1,1) node[right]{$e_1$};
        \draw (5,1) node[left]{$e_2$};
        \draw (5,5) node[left]{$e_3$};
        \draw (1,5) node[right]{$e_4$};
        
        \draw (3,2) node[below]{$e_5$};
        \draw (4,3) node[right]{$e_6$};
        \draw (3,4) node[above]{$e_7$};
        \draw (2,3) node[left]{$e_8$};
        \end{tikzpicture}
        \caption{$\cs_4$}
    \end{subfigure}
    \begin{subfigure}[b]{0.3\linewidth}
        \begin{tikzpicture}[scale = .5, thick]
        \draw (0,0)--(2,2);
        \draw (2,2)--(4,2);
        \draw (4,2)--(6,0);
        \draw (2,2)--(1,3);
        \draw (4,2)--(5,3);
        
        \draw [fill] (0,0) circle [radius = .1];
        \draw [fill] (1,1) circle [radius = .1];
        \draw [fill] (2,2) circle [radius = .1];
        \draw [fill] (1,3) circle [radius = .1];
        \draw [fill] (4,2) circle [radius = .1];
        \draw [fill] (5,3) circle [radius = .1];
        \draw [fill] (5,1) circle [radius = .1];
        \draw [fill] (6,0) circle [radius = .1];
        
        \draw (0,0) node[below left] {$1$};
        \draw (1,3) node[above left] {$2$};
        \draw (5,3) node[above right] {$3$};
        \draw (6,0) node[below right] {$4$};
        
        \draw (.25,.25) node[right] {$e_1$};
        \draw (1.25,2.75) node[right] {$e_2$};
        \draw (4.75,2.75) node[left] {$e_3$};
        \draw (5.75,.25) node[left] {$e_4$};
        \draw (1.25,1.25) node[right] {$e_5$};
        \draw (3,2) node[below] {$e_6$};
        \draw (4.75,1.25) node[left] {$e_7$};
        \end{tikzpicture}
        \caption{$\ct_1$}
    \end{subfigure}
    \begin{subfigure}[b]{0.3\linewidth}
        \begin{tikzpicture}[scale = .5, thick]
        
        \draw (0,0)--(2,2);
        \draw (2,2)--(4,2);
        \draw (4,2)--(6,0);
        \draw (2,2)--(1,3);
        \draw (4,2)--(5,3);
        
        \draw [fill] (0,0) circle [radius = .1];
        \draw [fill] (1,1) circle [radius = .1];
        \draw [fill] (2,2) circle [radius = .1];
        \draw [fill] (1,3) circle [radius = .1];
        \draw [fill] (4,2) circle [radius = .1];
        \draw [fill] (5,3) circle [radius = .1];
        \draw [fill] (5,1) circle [radius = .1];
        \draw [fill] (6,0) circle [radius = .1];
        
        \draw (0,0) node[below left] {$1$};
        \draw (1,3) node[above left] {$4$};
        \draw (5,3) node[above right] {$3$};
        \draw (6,0) node[below right] {$2$};
        
        \draw (.25,.25) node[right] {$e_1$};
        \draw (1.25,2.75) node[right] {$e_4$};
        \draw (4.75,2.75) node[left] {$e_3$};
        \draw (5.75,.25) node[left] {$e_2$};
        \draw (1.25,1.25) node[right] {$e_8$};
        \draw (3,2) node[below] {$e_7$};
        \draw (4.75,1.25) node[left] {$e_6$};
        \end{tikzpicture}
        \caption{$\ct_2$}
    \end{subfigure}
    \caption{The four leaf sunlet network $\cs_4$ and the two trees $\ct_1$ and $\ct_2$ that are obtained by deleting the reticulation edges $e_8$ and $e_5$ respectively.}
    \label{fig:4SunletAndTrees}
\end{figure}

Since the discrete Fourier transform is linear, it naturally extends to phylogenetic networks by \cref{eq:networkParam}. We end this subsection with the following example illustrating this new parameterization for both trees and networks.

\begin{example}\label{ex:4sunlet}
Let $\cs_4$ be the 4-sunlet pictured in Figure \ref{fig:4SunletAndTrees}(A). The two trees $\ct_1$ and $\ct_2$, pictured in Figure \ref{fig:4SunletAndTrees}(B) and (C), are obtained from $\cs_4$ by deleting the reticulation edges $e_8$ and $e_5$ respectively. 
By \cref{eq:networkParam}, the parameterization is given by 
\[\Psi_{\cs_4} = \lambda\Psi_{\ct_1} + (1-\lambda)\Psi_{\ct_2}.\]
where in the Fourier coordinates $\Psi_{\ct_1}$ and $\Psi_{\ct_2}$ are given by \cref{eqn:TreeParam}. Explicitly, we have
\[
q_{g_1 g_2 g_3 g_4} = \begin{cases}
\lambda a_{g_1}^1 a_{g_2}^2 a_{g_3}^3 a_{g_4}^4 a_{g_1}^5 a_{g_1 + g_2}^6 a_{g_4}^7 +
(1-\lambda)a_{g_1}^1 a_{g_2}^2 a_{g_3}^3 a_{g_4}^4 a_{g_2}^6 a_{g_1 + g_4}^7 a_{g_1}^8
& \mbox{ if } \sum_{i \in [4]}g_i = 0 \\ 
0 & \mbox{ otherwise.}
\end{cases}
\]
Here, the first term comes from the parameterization $\Psi_{\ct_1}$, and the second term comes from $\Psi_{\ct_2}$. Observe that for all $g$ in $G$, the parameters $a_g^5$ and $a_g^8$ only occur in the first and second terms, respectively, since they correspond to the deleted reticulation edges. This means we can make the substitution $\lambda a_g^5 \mapsto a_g^5$ and $(1-\lambda) a_g^8 \mapsto a_g^8 $ without changing the parameterization of the model as was noted in \cite{gross2018distinguishing}. So, for example, after this change, the coordinate $q_{1001}$ is given by
\[
q_{1001} = a_{1}^1 a_{0}^2 a_{0}^3 a_{1}^4 a_{1}^5 a_{1}^6 a_{1}^7 +
a_{1}^1 a_{0}^2 a_{0}^3 a_{1}^4 a_{0}^6 a_{0}^7 a_{1}^8 = a_{1}^1 a_{0}^2 a_{0}^3 a_{1}^4 (a_{1}^5 a_{1}^6 a_{1}^7 + a_{0}^6 a_{0}^7 a_{1}^8). 
\]
\end{example}

\subsection{Phylogenetic Invariants and Algebraic Statistics}\label{subsec: phylogenetic invariants}
In this subsection we give a brief overview of some algebraic structures which we will need to describe the invariants for the CFN model on phylogenetic trees \cite{sturmfels2005toric} and the quadratic invariants for the CFN model on phylogenetic networks that were found in \cite{cummings2024invariants}. Both of these characterizations will be key in the construction of our alternative parameterization of the CFN sunlet model in Section \ref{sec:pfaffian} and our description of all of the invariants in this new coordinate system. For additional information on the algebraic concepts discussed below, we refer the reader to \cite{coxlittleoshea, algstat}. 

As discussed in the previous section, all phylogenetic Markov models on a tree or a network can be described in terms of the image of a polynomial map $\varphi$. A fundamental problem concerning any such phylogenetic model is to determine the \emph{ideal of phylogenetic invariants}, also called the \emph{vanishing ideal} of the model. 

\begin{definition}
\label{defn:VanishingIdeal}
Let $S \subseteq \Cc^m$. The \emph{vanishing ideal} of $S$
is the set of multivariate polynomials
\[
\ci(S) := \{f \in \Cc[x_1, \ldots, x_m] ~|~ f(s) = 0 ~\mbox{for all}~ s \in S\}.
\]
\end{definition}

For instance, if $\varphi = \Psi_\ct$ is the parameterization map  in Fourier coordinates of the CFN model on a tree $\ct$, then the ideal of phylogenetic invariants is $I_\ct = \ci(\image(\Psi_\ct))$. This means that phylogenetic invariants are the multivariate polynomials in the Fourier coordinates $q_{g_1 \ldots g_n}$ that evaluate to zero on every point in the image of the map $\Psi_\ct$. Since $q_{g_1\ldots g_n} = 0$ when $g_1 + \cdots + g_n \neq 0$, we think of the ideal $I_{\ct}$ (and thus also the invariants) as living in the polynomial ring
\[
Q_n = \mathbb{C}[q_{g_1\ldots g_n}\,|\, g_1 + \cdots + g_n = 0].
\]
We can similarly define the ideal of invariants for a network as $I_\cn = \ci(\image(\Psi_\cn))$. 

Every ideal $I$ of a polynomial ring $\Cc[x_1, \ldots, x_m]$ has an associated \emph{affine algebraic variety} which is
\[
\VV(I) = \{s \in \Cc^m ~|~ f(s) = 0 ~\mbox{for all}~ f \in I\}. 
\]
When $I = \mathcal{I}(S)$ is the vanishing ideal of a set $S$, then $\VV(I)$ is called the \emph{Zariski closure} of $S$. As a set, it contains $S$, but may be larger. The \emph{phylogenetic variety} of a tree $\ct$ is $\VV(I_\ct)$, and we can define the phylogenetic variety of a network similarly. For brevity, we denote $\VV(I_\ct)$ by $\VV_\ct$. 

The Zariski closure of the image of $\Psi_\ct$ contains the Fourier transform of the model, $M_\ct$, but they are not equal. By viewing $\Psi_{\ct}$ as a complex polynomial map we are allowing parameters to take any values in $\mathbb{C} $ (rather than probabilities between 0 and 1), and forgetting the stochastic restrictions imposed by the Markov process (such as that the sum of each row must equal 1). Thus, in a sense, $\VV_\ct$ is much larger than the original model. In particular, for a point $v \in \VV_\ct$, it is not necessarily the case that its inverse image under the Fourier transformation is a probability distribution, and therefore lies in the probability simplex $\Delta_{\kappa^n-1}$. We call the Fourier transform of $M_\ct$ in $\VV_\ct$ the \emph{stochastic} part of the variety. Observe that by equation (\ref{eqn:fouriercoords}) we have
\[ 
q_{0\ldots 0} =  \sum_{(h_1,\ldots, h_n) \in (\mathbb{Z}/2\mathbb{Z})^n}p_{h_1,\ldots, h_n},
\]
so the stochastic part of the variety is contained in the open set given by $q_{0\ldots 0} \neq 0$. In fact, this is true more generally for any abelian group $G$.

While we have defined the phylogenetic variety as an affine variety, it is common to think of these as \emph{projective varieties}. Indeed, the parameterization map $\Psi_{\ct}$ is homogeneous (in the sense of total degree). Thus, if $p$ is a point in the stochastic part of $\VV_{\ct}$ representing a probability distribution, then $\lambda \cdot p$ will lie in the phylogenetic variety for any $\lambda \in \Cc \setminus\{0\}$; however, in projective space, these points are identified. Therefore, when we refer to the phylogenetic variety from here on out, we mean the Zariski closure of the image of $\Psi_\ct$ in $\Pp^{2^{n-1}-1}$. This is desirable, but it has the side effect of adding points at infinity, i.e., there will be points in the projective variety where the sum of the probability coordinates is 0, and these cannot correspond to a probability distribution. As such, we define \emph{the statistically relevant affine open patch} of the projective variety $\VV_\ct$ as $\VV_\ct \setminus \VV(\sum_{(g_1,\ldots,g_n)\in\mathbb{Z}/2\mathbb{Z}} p_{g_1\ldots g_n}$).


Recall that points in (complex) projective space $\Pp^n$ are given by homogeneous coordinates $[x_0:x_1:\ldots:x_n]$ where each $x_i\in \Cc$ and for all $\lambda\in\Cc\setminus\{0\}$ we have $[x_0:x_1:\ldots:x_n] = [\lambda x_0:\lambda x_1:\ldots: \lambda x_n]$. The reader may consult, for example \cite{coxlittleoshea}, for more details.

\begin{remark}\label{rem: group-based models are projective}
   The CFN parameterization of the $n$-sunlet $\cs_n$ can be thought of as the image of a map of the form
    \[
        \large(\prod_{e \in E(\cs_n)\setminus \{e_{n+1}, e_{2n}\}} \Pp^1 \large)\times \Pp^3 \to \Pp^{2^{n-1} - 1},
    \]
   where $e_{n+1}$ and $e_{2n}$ are the reticulation edges, the homogeneous coordinates of $\Pp^3$ are $[a_{n+1}^0:a_{n+1}^1:a_{2n}^0:a_{2n}^1]$, the homogeneous coordinates of each $\Pp^1$ are $[a_0^e : a_1^e]$, and the homogeneous coordinates of $\Pp^{2^{n-1} - 1}$ are the Fourier coordinates $[q_{0\ldots 0}:\ldots : q_{1\ldots 1}]$.
   Thus we may think of $\VV_\cn$ as a projective variety in $\Pp^{2^{n-1} - 1}$. In \cref{sec:pfaffian}, we will consider the statistically relevant affine open patch of this variety where $\sum_{g \in (\zmodtwo)^n} p_g = q_{00\dots0} = 1$. 
   
\end{remark}

For a phylogenetic tree $\ct$ under the CFN model, $I_{\ct}$ has a particularly nice description in terms of the minors of certain matrices. Let $\ct$ be a phylogenetic tree on $n$ leaves and $e$ an edge of $\ct$. In Fourier coordinates, $I_{\ct}$, the ideal of the phylogenetic invariants on $\ct$, is contained in the polynomial ring $Q=\Cc[q_{g_1\ldots g_n}\,|\,g_1 + \cdots + g_n = 0]$. Let $A|B$ be the split induced by the edge $e$, and write $A=\{A_1,\ldots, A_m\}\subset [n]$ and $B = \{B_1\ldots,B_{n-m}\}\subset [n]$, where $A_i < A_j$ for $i <j$ and $B_i < B_j$ for $i <j$.
The indeterminates of $Q$ can be partitioned into two disjoint sets $\{q_{\mathbf{g}}\,|\,\sum_{j\in A}g_j=0\in \mathbb{Z}/2\mathbb{Z}\}$ and $\{q_{\mathbf{g}}\,|\,\sum_{j\in A}g_j=1\in \mathbb{Z}/2\mathbb{Z}\}$, where we write $\mathbf{g}=(g_1,\ldots,g_n)$. These sets fit into two matrices, $M_0^e$ and $M_1^e$ of dimension $2^{|A|-1}\times 2^{|B|-1}$, whose rows and columns are indexed by the substrings $\mathbf{g}_A = (g_{A_1},\ldots, g_{A_m})$ and $\mathbf{g}_B = (g_{B_1},\ldots, g_{B_{n-m}})$ of $(g_1,\ldots, g_n)$, respectively. Using this notation, the ideal $I_{\ct}$ for any tree $\ct$ under the CFN model is completely characterized by the following theorem from \cite{sturmfels2005toric}. 

\begin{theorem}\cite{sturmfels2005toric}
    Let $\ct$ be a phylogenetic tree and $I_\ct$ be the ideal of phylogenetic invariants of $\ct$ under the CFN model. Then $I_\ct$ is generated by the $2 \times 2$ minors of all matrices of the form $M_0^{A_e|B_e}$ and $M_1^{A_e|B_e}$ where $e$ is an internal edge. 
\end{theorem}

The following example illustrates this construction for a small tree. 

\begin{example}
\label{ex:cfnTreeIdeal}
Let $\ct$ be the unrooted binary tree with 4 leaves whose only nontrivial split is $A|B = 12 | 34$. Then $I_\ct$ is generated by the $2 \times 2$ minors of the following matrices
\[
M_0^{A|B} = 
\begin{blockarray}{ccc}
      & 00 & 11 \\
      \begin{block}{c(cc)}
      00 & q_{0000} & q_{0011}\\
      11 & q_{1100} & q_{1111}\\
      \end{block}
\end{blockarray}
 ~~~~~~~~
M_1^{A|B} = 
\begin{blockarray}{ccc}
      & 01 & 10 \\
      \begin{block}{c(cc)}
      01 & q_{0101} & q_{0110}\\
      10 & q_{1001} & q_{1010}\\
      \end{block}
\end{blockarray}.
\]
So the ideal of phylogenetic invariants for $\ct$ is $I_\ct = \langle q_{0000}q_{1111}-q_{0011}q_{1100},~ q_{0101}q_{1010} - q_{0110}q_{1001} \rangle$. 
\end{example}

In \cite{cummings2024invariants}, the authors show that the invariants of any level-1 phylogenetic network under a group-based model can be determined from the invariants of phylogenetic trees and sunlet networks using the toric fiber product \cite{sullivant2007toric}. They then give a complete description of the quadratic invariants of the model and conjecture that the ideal $I_{\cs_n}$, where $\cs_n$ is the sunlet network on $n$ leaves, is generated by these quadratics. One key observation from \cite{cummings2024invariants} is that $I_{\cs_n}$ is homogeneous in the multigrading defined by
\begin{equation}\label{eqn:multigrading}
    \deg(q_{\mathbf{g}}) = (1,\mathbf{g}) \in \Zz^{n+1},
\end{equation}
where $\mathbf{g}=(g_1,\ldots, g_n)$ is treated as a $0/1$ integer vector instead of as a vector with $\zmodtwo$ entries. In the next section, we will provide an alternate perspective on the ideal $I_{\cs_n}$, which lets us describe the invariants of this model using determinantal constraints. 

In \cref{sec:pfaffian}, it will be helpful for us to know the dimensions of the phylogenetic varieties of sunlet networks $\cn$ under the CFN model. For all group-based models on level-1 phylogenetic networks, the dimension of the phylogenetic variety is studied in \cite{gross2023dimensions}. We will use the following.

\begin{proposition}\cite[Proposition~14]{gross2023dimensions}\label{prop:sunletDimension}
Let $\cs_n$ be an $n$-sunlet network with $n\geq 5$ under the CFN model. Then, the dimension of the projective variety $\VV(I_{\cs_n}) \subseteq \Pp^{2^{n-1}-1}$ is $2n - 1$.
\end{proposition}

We end this section with a discussion about the discrete Fourier transform. As it turns out, this is more than just an algebraic ``trick", but rather, the discrete Fourier transform allows us to view the variety in terms of moments rather than through probability distributions. Similar analyses can also be found in \cite[\S 3]{evans1993invariants} and \cite[\S 2]{sturmfels2005toric}.

Recall that
\[
    q_{g_1\dotsc g_n} = \sum_{(h_1,\dotsc,h_n) \in (\zmodtwo)^n} (-1)^{\sum_{i=1}^n g_i h_i}\ p_{h_1\ldots h_n}.
\]
Our first observation is that, for any phylogenetic network, the Fourier coordinate $q_{g_1,\dotsc,g_n}$ is an expected value. To see this, consider the character $\chi : \zmodtwo \to \Rr$,
\[
     \chi(x) = \begin{cases}
        -1 &\text{if } x = 1 \\
        1  &\text{if } x = 0
    \end{cases},
\]
and define $Y_i = \chi(X_i)$ where $X_1,\dotsc,X_n$ are the observable random variables at the leaves of a phylogenetic tree. Then, the following theorem holds. This theorem follows directly from the discussion in  \cite[\S 3]{evans1993invariants}; however, we present a proof here for completeness.

\begin{theorem}
\label{thm: fourier coordinates are expected values}
    For any distribution of $n$ binary random variables taking values in $\{-1,1\}$, the Fourier coordinates are moments. With the notation above,
    \[
        q_{g_1,\dotsc,g_n} = \mathbb{E}\left[ \prod_{g_i = 1} Y_i \right].
    \]
    Moreover, for the CFN model, all the odd moments vanish, i.e., if $\sum_i g_i = 1$, then $q_{g_1,\dotsc,g_n} = 0$.
\end{theorem}
\begin{proof}
    For $g,h \in (\zmodtwo)^n$, we claim that 
    \[
        (-1)^{\sum_{i=1}^n g_i h_i} = \prod_{g_i = 1} \chi(h_i).
    \]
    Indeed, if we rewrite the product on the left-hand side as $\prod_{i=1}^n (-1)^{g_i h_i}$, we see that the $i^\text{th}$ factor is $-1$ if and only if $g_i = h_i = 1$. There are precisely the same number of factors of $-1$ on the right-hand side. Thus, we arrive at the desired formula for the Fourier coordinate $q_{g_1,\dotsc,g_n}$.
    \begin{align*}
        \mathbb{E}\left[\prod_{g_i = 1} Y_i \right] &= \sum_{h \in (\zmodtwo)^n} \left(\prod_{g_i = 1} \chi(h_i)\right) p_{h_1,\dotsc,h_n} \\
        &= \sum_{h \in (\zmodtwo)^n} (-1)^{\sum_{i=1}^n g_i h_i} p_{h_1,\dotsc,h_n} \\
        &= q_{g_1,\dotsc,g_n}
    \end{align*}

    The second statement can be deduced from the fact that $p_{g_1,\dotsc,g_n} = p_{g_1 + 1,\dotsc,g_n + 1}$. Consider a tree $\ct$; the argument for networks follows easily from the tree case and \cref{eq:networkParam}. In general, we have that
    \begin{equation}
    \label{eq:prob}
    p_{g_1,\dotsc,g_n} = \sum_{x \in X(g_1,\ldots, g_n)}\pi_{x_\rho}\prod_{(u,v) \in E(\ct)}M_{x_u, x_v}^{(u,v)}.
    \end{equation}
    However, since we are restricting to the CFN model, we know that $\pi_{x_\rho} = 1/2$ and there are functions $f_e:\zmodtwo \to \Rr$ for all $e = (u,v) \in E(\ct)$ such that $M_{x_u,x_v}^e = f_e(x_u - x_v)$. Specifically, $f_e$ is defined as follows.
        \[f_e(x) = \begin{cases} \alpha_e &\text{if }x = 0 \\ \beta_e &\text{if } x = 1\end{cases}\]
    Thus, we can rewrite \cref{eq:prob} as
    \begin{equation}
    \label{eq:probFourier}
        p_{g_1,\dotsc,g_n} = \frac{1}{2}\sum_{x \in X(g_1,\ldots, g_n)}\prod_{e =(u,v) \in E(\ct)}f_e(x_u - x_v).
    \end{equation}
    The key observation is that there is an involution, $\sigma : X(g_1,\dotsc,g_n) \to X(g_1,\dotsc,g_n)$ where $\sigma(x)$ has the same values at the leaves as $x$, but the value at each internal node is switched from 0 to 1 or vice versa. This has no affect on the right hand side of \cref{eq:probFourier} other than reordering the sum. Now, suppose $j$ is a leaf and its parent is $u_j$. In the product, we have the factor
    $f_e(\sigma(x_{u_j}) - g_j) = f_e(x_{u_j} - (g_j + 1))$. Following definitions, we see that the sum could have been taken over $X(g_1+1,\dotsc,g_n+1)$ instead. It follows that $p_{g_1,\dotsc,g_n} = p_{g_1+1,\dotsc,g_n+1}$.

    Now, we can use this fact to prove that if $\#\{i ~:~ g_i = 1\}$ is odd then $q_{g_1,\dotsc,g_n} = 0$. Indeed, by reindexing the sum in the definition of the discrete Fourier transform, we have
    \[
        q_{g_1,\dotsc,g_n} = \sum_{h \in (\zmodtwo)^n} (-1)^{\sum_{i=1}^n g_ih_i}p_h = \sum_{h + \mathbf{1} \in (\zmodtwo)^n} (-1)^{\sum_{i=1}^n g_i(h_i+1)}p_{h + \mathbf{1}}.
    \]
    Since $\#\{i ~:~ g_i = 1\}$ is odd, we see that $(-1)^{\sum_{i=1}^n g_i(h_i + 1)} = -(-1)^{\sum_{i=1}^n g_i h_i}$. Making this simplification on the far right hand side and the substitution $p_h = p_{h+\mathbf{1}}$, we see that $q_{g_1,\dotsc,g_n}$ must be identically 0.
   
\end{proof}

\subsection{Pfaffians of Skew-Symmetric Matrices}
In this subsection, we briefly remind the reader of some properties of the Pfaffian of a skew-symmetric matrix, which will be crucial in \cref{sec:pfaffian}. 

\begin{definition}\label{Pfaffian of a A}
Let $S_n$ be the symmetric group of order $n$, and $A = (a_{ij})$ an $n\times n$ skew-symmetric matrix where $n=2m$ for some positive integer $m$. The Pfaffian of $A$, denoted $\text{Pf}(A)$, is defined as 
\begin{equation}
    \text{Pf}(A)=\frac{1}{2^mm!}\sum_{\sigma\in S_n}\sign(\sigma)\prod_{1\leq j\leq m}a_{\sigma(2j-1),\sigma(2j)}.
\end{equation}
\end{definition}

The following proposition, originally due to Cayley in 1849 \cite{Cayley1849}, gives us two ways to compute the Pfaffian of a skew-symmetric matrix, the second of which will be of use to us in \cref{sec:pfaffian}.

\begin{proposition}
\label{prop: pfaffian laplace expansion}
    Let $A = (a_{ij})$ be a skew-symmetric matrix. Then
    \[(\Pf (A))^2=\det(A).\]
    Moreover, there is also a Laplace-like expansion for the Pfaffian. If $A$ is a $2m\times 2m$ skew-symmetric matrix, let $A_{ij}$ be the submatrix of $A$ where we delete both the $i^\text{th}$ and $j^\text{th}$ rows and columns of $A$. Then 
\begin{equation}\label{eqn:laplace}
    \mathrm{Pf}(A) = \sum_{j=2}^{2m} (-1)^j a_{1j}\mathrm{Pf}(A_{1j}).
\end{equation}
\end{proposition}

\section{The Pfaffian Structure of CFN Network Ideals} \label{sec:3}
In this section, we study the vanishing ideal of the CFN sunlet network model. In \Cref{sec:pfaffian}, we find the vanishing ideal of the CFN sunlet network model in a new coordinate system, and prove the main theorem from the introduction. In \Cref{sec:quadratic generation}, we study the vanishing ideal of the model in Fourier coordinates, and partially resolve a conjecture from \cite{cummings2024invariants}. Throughout, we assume that $n\geq 4$, unless otherwise stated, and we denote the space of $n\times n$ skew-symmetric matrices with entries in $\Cc$ by $\Skew_n(\Cc)$.

\subsection{Proof of the Main Theorem}
\label{sec:pfaffian}
In this subsection, we show that the statistically relevant affine open patch of the CFN sunlet network model factors through the space of skew-symmetric matrices via the Pfaffian. We then use this to prove the following theorem, which is the main theorem of our paper.

\begin{theorem}\label{thm:main}
   Let $\VV_{\cs_n} \subseteq \Pp^{2^{n-1} - 1}$ be the \emph{projective} phylogenetic variety associated to the CFN model on a $n$-leaf sunlet network for $n \geq 4$. 
   Denote the affine open patch of $\VV_{\cs_n}$ where the sum of the coordinates is 1 by $U_n$.
    Then:
    \begin{enumerate}
        \item [(a)] There is an embedding $\pi : U_n \to \Skew_n(\Cc)$.
        \item [(b)] Let $\Omega$ be a $n\times n$ skew-symmetric matrix with variable entries $x_{ij}$. In these coordinates, $U_n$ is cut out by the following minors of $\Omega$:
        \begin{enumerate}
            \item [(i)] $\det(\Omega_{\{i,j\},\{k,\ell\}})$ where $1 < i < j < k < \ell \leq n$ and 
            \item [(ii)] $\det(\Omega_{\{1,i_2,i_3\}, \{j_1,j_2,j_3\}})$ where $1 < i_2 < j_1 < j_2 < j_3 < i_3 \leq n$,
        \end{enumerate}
        where $\Omega_{S,T}$ is the submatrix of $\Sigma$ whose rows and columns are indexed by the elements of $S$ and $T$, respectively.
    \end{enumerate} 
\end{theorem}

To prove this theorem, we first construct the map $\pi: U_n \to \Skew_n(\Cc)$ and show that it is an embedding. We then show that the minors described in Theorem {\ref{thm:main}} parts (b)(i) and (b)(ii) all lie in the vanishing ideal $\ci(\pi(U_n))$ and form a Gr\"obner basis for the ideal that they generate. Note that when $n=5$ there is only a single $2\times 2$ minor in part (b)(i) and no minors from part (b)(ii), and when $n=4$ there are no minors from neither part (b)(i) nor part (b)(ii). Lastly, we show that this ideal is prime and of the correct dimension, and thus, the minors also form a Gr\"obner basis for $\ci(\pi(U_n))$ which completes the proof. Before presenting the proof, we begin with a motivating example for $n=4$.

\begin{example}
\label{ex: extended example}
    Let $\VV_{\cs_4} \subseteq \Pp^7$ be the CFN phylogenetic variety of the 4-leaf sunlet network. For 4 leaves we have $2^{3} = 8$ non-zero Fourier coordinates, so $\VV_{\cs_4}$ lives in $\Pp^7$ (see Remark {\ref{rem: group-based models are projective}}). The statistically relevant affine open patch $U_4$ of $\VV_{\cs_4}$ is $U_4 = \VV_{\cs_4} \setminus \VV(q_{0000}) \subseteq \Cc^7$, i.e. we set $q_{0000} = 1$.
    We claim that the map $\pi : U_4 \to \Skew_4(\Cc)$ defined as follows is an embedding:
    \begin{equation}\label{eqn:u4example}
        (q_{1100}, q_{1010}, q_{1001}, q_{0110},q_{0101},q_{0011},q_{1111}) \mapsto \begin{pmatrix}
            0 & q_{1100} & q_{1010} & q_{1001} \\
            -q_{1100} & 0 & q_{0110} & q_{0101} \\
            -q_{1010} & -q_{0110} & 0 & q_{0011} \\
            -q_{1001} & -q_{0101} & -q_{0011} & 0
        \end{pmatrix}.
    \end{equation}
    The map $\pi$ is a projection onto those Fourier coordinates consisting of exactly two $1$'s, so it remains for us to show that given a $4\times 4$ skew-symmetric matrix in the image of $\pi$, we can recover $q_{1111}$ uniquely. If $q_{1111}$ were arbitrary, then, of course, this would be impossible; however, $\VV_{\cs_4}$ is defined by the quadric below:
    \begin{equation}\label{eqn:S4relation}
        \VV_{\cs_4} = \VV(q_{0000}q_{1111} - q_{0011}q_{1100} + q_{0101}q_{1010} - q_{0110}q_{1001}).
    \end{equation}
    Since $q_{0000} = 1$ on $U_4$, we can solve for $q_{1111}$ in terms of the remaining variables on this open patch.
    It follows that $\pi$, in this case, is an isomorphism whose inverse is given below.
    \[
        \begin{pmatrix}
        0 & x_{12} & x_{13} & x_{14} \\
        -x_{12} & 0 & x_{23} & x_{24} \\
        -x_{13} & -x_{23} & 0 & x_{34} \\
        -x_{14} & -x_{24} & -x_{34} & 0
        \end{pmatrix}
        \mapsto 
        (x_{12}, x_{13}, x_{14}, x_{23}, x_{24}, x_{34}, x_{12}x_{34} - x_{13}x_{24} + x_{14}x_{23})
    \]
    Observe that $x_{ij}$ corresponds to the $q_{g_1g_2g_3g_4}$ for which $g_i = g_j = 1$ and $g_k = 0$ for $k\neq i,j$. The inverse $\pi^{-1} : \Skew_4(\Cc) \to U_4$ can be written compactly using Pfaffians.
    Indeed, let $\Omega  = (x_{ij}) \in \Skew_4(\Cc)$. For a subset $S \subseteq [4]$ of even cardinality, we define $\Omega_S$ as the matrix whose rows and columns are indexed by $S$, and we can compute their Pfaffians.
    \begin{align*}
        \mathrm{Pf}(\Omega_\emptyset) &= 1 &  
        \mathrm{Pf}(\Omega_{\{1,2\}}) &= x_{12} \\
        \mathrm{Pf}(\Omega_{\{1,3\}}) &= x_{13} &  
        \mathrm{Pf}(\Omega_{\{1,4\}}) &= x_{14} \\
        \mathrm{Pf}(\Omega_{\{2,3\}}) &= x_{23} & 
        \mathrm{Pf}(\Omega_{\{2,4\}}) &= x_{24} \\
        \mathrm{Pf}(\Omega_{\{3,4\}}) &= x_{34} & 
        \mathrm{Pf}(\Omega) &=\mathrm{Pf}(\Omega_{[4]})= x_{12}x_{34} - x_{13}x_{24} + x_{14}x_{23}
    \end{align*}
    The polynomials above satisfy the homogeneous relation below:
    \[
        \mathrm{Pf}(\Omega_\emptyset)\mathrm{Pf}(\Omega) - \mathrm{Pf}(\Omega_{\{1,2\}})\mathrm{Pf}(\Omega_{\{3,4\}}) + \mathrm{Pf}(\Omega_{\{1,3\}})\mathrm{Pf}(\Omega_{\{2,4\}}) - \mathrm{Pf}(\Omega_{\{1,4\}})\mathrm{Pf}(\Omega_{\{2,3\}}) = 0.
    \]
    This mirrors the relation in (\ref{eqn:S4relation}), which defines $\VV_{\cs_4}$ in $\Pp^7$.
    The inverse of $\pi$ can then be rewritten as 
    \begin{align*}\pi^{-1} : \Skew_4(\Cc) &\to U_4 \\ \Omega &\mapsto (\mathrm{Pf}(\Omega_S))_{S \in \mathcal{A}}\end{align*}
    where $\mathcal{A}$ is the set of all subsets of $[4]$ of even cardinality. Here, we think of elements in $U_4$ as being vectors in $\mathbb{C}^7$, indexed by Fourier coordinates (except $q_{0000})$. Each element of $\mathcal{A}$ corresponds to some Fourier coordinate $q_{g_1 g_2 g_3 g_4}$, and gives the $i$ for which $g_i = 1$ in $q_{g_1 g_2 g_3 g_4}$, and are ordered as in equation ({\ref{eqn:u4example}}). This shows that $U_4$ embeds in $\Skew_4(\Cc)$ and the inverse is given by Pfaffians.
\end{example}


Our first result generalizes \cref{ex: extended example} to all $n \geq 4$. We write $\mathbf{g}=(g_1,\ldots,g_n)\in(\zmodtwo)^n$. Let $\VV_{\cs_n}$ denote the phylogenetic variety of the $n$-sunlet $\cs_n$ in Fourier coordinates. The vanishing ideal of $\VV_{\cs_n}$ is the kernel of the following homomorphism of polynomial rings:
\begin{align*}
    \varphi_n : Q_n &\to A_n \\
                q_{\mathbf{g}} &\mapsto \prod_{i=1}^n a_{g_i}^i \left(\prod_{i=1}^{n-1} a^{n+i}_{\sum_{j=1}^i g_j} + \prod_{i=2}^{n} a^{n+i}_{\sum_{j=2}^i g_j}  \right)
\end{align*}
where
\begin{align*}
    Q_n &= \Cc\left[q_{\mathbf{g}} ~|~ \mathbf{g} = (g_1,\ldots, g_n) \in (\zmodtwo)^n \text{ where } \sum_{i=1}^n g_i = 0\right] \text{ and } \\
    A_n &= \Cc[a^i_g ~|~ 1 \leq i \leq 2n, \;g\in \zmodtwo].
\end{align*}

We are interested in the affine open patch where $q_{\mathbf{0}} \neq 0$, where now we write $\mathbf{0}=(0,\ldots,0)\in (\zmodtwo)^n$ for brevity. As a reminder, this is because $q_{\mathbf{0}}$ is the sum of the probability coordinates, so looking at this open patch amounts to enforcing the extra condition that the sum of the probability coordinates is 1 as we are free to scale $q_{\mathbf{0}}$ to any non-zero complex number on this open patch. As such, we let $U_n$ be $\VV_{\cs_n} \setminus\VV (q_{\mathbf{0}})$, and we refer to $U_n$ as the \emph{statistically relevant affine open patch}. Then the vanishing ideal of $U_n$ is given by the kernel of the induced map defined below.
\begin{align*}
(\varphi_n)_0 : ((Q_n)_{q_{\mathbf{0}}})_0 &\to ((A_n)_{\varphi_n(q_{\mathbf{0}})})_0 \\
\frac{q_{\mathbf{g}}}{q_{\mathbf{0}}} &\mapsto  \frac{\varphi_n(q_{\mathbf{g}})}{\varphi_n(q_{\mathbf{0}})}
\end{align*}
Note that $((Q_n)_{q_{\mathbf{0}}})_0$ refers to localizing $Q_n$ at $q_{\mathbf{0}}$ and then taking the degree 0 part of the ring, i.e. it is the polynomial ring in the variables $q_{\mathbf{g}}/q_{\mathbf{0}}$. Similarly, $((A_n)_{\varphi_n(q_{\mathbf{0}})})_0$ is the degree 0 part of the ring $A_n$ localized at $\varphi_n(q_{\mathbf{0}})$. Throughout the rest of this section, we denote the kernel of this map as
\[
    J_n = \ker((\varphi_n)_0).
\]

Now, define $X_n$ to be the polynomial ring $\Cc\left[x_{ij} ~|~ 1 \leq i < j \leq n\right]$. Let $\Omega^n = (x_{ij})$ be an $n\times n$ skew-symmetric matrix with variable entries, and let $\{e_i\}_{i=1}^{n}$ be the standard basis for the $\mathbb{Z}$-module $(\zmodtwo)^n$. Define the ring homomorphism
\begin{align*}
    \psi_n : ((Q_n)_{q_{\mathbf{0}}})_0 &\rightarrow X_n \\
            \frac{q_{\mathbf{g}}}{q_{\mathbf{0}}} &\mapsto \mathrm{Pf}(\Omega^n_\mathbf{g})
\end{align*}
where $\mathbf{g} = \sum_{j=1}^{2k} e_{i_j} \in (\zmodtwo)^n$ for some $k\in\Nn$ with $2k\leq n$, and some set $\{i_1,\ldots, i_{2k}\} \subset [n]$, and where $\Omega^n_\mathbf{g}$ is the $2k\times 2k$ skew-symmetric sub-matrix of $\Omega^n$ whose rows and columns are indexed by the support of $\mathbf{g}$. Define 
\begin{align*}
    \widehat{\varphi_n} : X_n &\rightarrow ((A_n)_{\varphi_n(q_{\mathbf{0}})})_0 \\
            x_{ij} &\mapsto \frac{\varphi_n(q_{e_i + e_j})}{\varphi_n(q_{\mathbf{0}})}
\end{align*}
and denote the kernel of this map by
\[
    I_n = \ker(\widehat{\varphi_n}).
\]

\begin{proposition}
    \label{prop:FactorThroughPfaffian}
    The map $(\varphi_n)_0$ factors through $\psi_n$, i.e. $(\varphi_n)_0 = \widehat{\varphi_n} \circ \psi_n$. Moreover, there is an isomorphism $((Q_n)_{q_{\mathbf{0}}})_0/J_n \isom X_n/I_n$; thus, $U_n$ and $\VV(I_n) \subseteq\Skew_n(\Cc)$ are isomorphic.
\end{proposition}

The following lemma is key to the proof of \cref{prop:FactorThroughPfaffian}

\begin{lemma}
\label{lemma: PfaffianExpansion}
    Let $\{e_1,\dotsc,e_n\}$ be the standard basis for $(\zmodtwo)^n$. For $\mathbf{g} = \sum_{j=1}^{2k} e_{i_j}$ where $i_1 < i_2 < \dotsc < i_{2k}$ and $k \geq 2$, the polynomial,
    \[
       F = q_{\mathbf{0}}q_{\mathbf{g}} - \sum_{j=2}^{2k} (-1)^j q_{e_{i_1}+ e_{i_j}} q_{\mathbf{g} + e_{i_1}+ e_{i_j}},
    \]
    lies in the kernel of $\varphi_n$ and is homogeneous of degree $(2,\mathbf{g})$ in the multigrading defined in equation ({\ref{eqn:multigrading}}).
\end{lemma}

\begin{proof}
    The proof follows directly from \cite[Theorem 4.5]{cummings2024invariants}. We recall the necessary details here.
    By \cite[Theorem 4.5]{cummings2024invariants}, $F$ is an invariant if and only if it is in the kernel of two linear maps, $M_\mathbb{E}$ and $M_\mathbb{O}$, defined on the linear subspace of $Q_n$ with multidegree $(2,\mathbf{g})$,
    \[
    (Q_n)_{(2,\mathbf{g})} = \mathrm{span}_\Cc\{q_{\mathbf{h}}q_{\mathbf{k}} ~|~ \deg(q_{\mathbf{h}}q_{\mathbf{k}}) = (2,\mathbf{g})\}.
    \] 
Define $\mathcal{F} = [n] \setminus \{i_1,\dotsc,i_{2k}\}$ to be the complement of the support of $\mathbf{g}$, and consider the sets below.
    \begin{align*}
        \mathbb{E} &= \{i ~|~ |[i] \setminus \mathcal{F}| \text{ is even and } 2 \leq i \leq n-1\} \\
        \mathbb{O} &= \{i ~|~ |[i] \setminus \mathcal{F}| \text{ is odd and } 2 \leq i \leq n-1\}
    \end{align*}

    Now, if $\mathbf{h} <_\text{lex} \mathbf{k}$ and the degree of $q_{\mathbf{h}}q_{\mathbf{k}}$ is $(2,\mathbf{g})$, then we define the two linear maps via
    \begin{align*}
        M_\mathbb{E}(q_{\mathbf{h}} q_{\mathbf{k}}) &= \left(\sum_{i=1}^l h_i \right)_{l\in \mathbb{E}} \in (\zmodtwo)^\mathbb{E}, \\
        M_\mathbb{O}(q_{\mathbf{h}} q_{\mathbf{k}}) &= \left(\sum_{i=1}^l h_i \right)_{l\in \mathbb{O}} \in (\zmodtwo)^\mathbb{O},
    \end{align*}
    where $\mathbf{h}=(h_1,\ldots,h_n)$ and $\mathbf{k}=(k_1,\ldots,k_n)$. If $j$ is even, then
    \begin{align*}
        M_{\mathbb{E}} (\mathbf{q}_{\mathbf{g}+e_{i_1}+ e_{i_j}} q_{e_{i_1}+ e_{i_j}}) &= M_\mathbb{E} (q_{\mathbf{g}+e_{i_1}+ e_{i_{j-1}}} q_{e_{i_1}+ e_{i_{j-1}}}), \\
        M_{\mathbb{O}} (q_{\mathbf{g}+e_{i_1}+ e_{i_j}} q_{e_{i_1}+ e_{i_j}}) &= M_\mathbb{O} (q_{\mathbf{g}+e_{i_1}+ e_{i_{j+1}}} q_{e_{i_1}+ e_{i_{j+1}}}). 
    \end{align*}
    In particular, we see that $F$ is in $\ker M_\mathbb{E} \cap \ker M_\mathbb{O}$; hence, $F \in \ker \varphi_n$ as claimed.
\end{proof}

\begin{proof}[Proof of \cref{prop:FactorThroughPfaffian}]
    The second statement follows from the first by the first isomorphism theorem. For the first statement, we must show that for $\mathbf{g} = \sum_{j=1}^{2k} e_{i_j}$ we have
    \[(\varphi_n)_0(q_{\mathbf{g}}/q_{\mathbf{0}}) = \widehat{\varphi_n}(\psi_n(q_{\mathbf{g}}/q_{\mathbf{0}})).\]  
    
    We proceed by induction on $k$.
    When $k = 1$, we have $\mathbf{g} = e_{i_1} + e_{i_2}$ for some $i_1 < i_2$, so
    \[
        \widehat{\varphi_n}(\psi_n(q_{e_{i_1} + e_{i_2}}/q_{\mathbf{0}})) = \widehat{\varphi_n}(x_{ij}) = \frac{\varphi_n(q_{e_{i_1} + e_{i_2}})}{\varphi_n(q_{\mathbf{0}})} = (\varphi_n)_0(q_{e_{i_1}+e_{i_2}}/q_{\mathbf{0}})
    \]
    establishing the base case for the induction.
    
    Now, we assume that $k \geq 2$.
    \begin{align*}
        \widehat{\varphi_n}(\psi_n(q_{\mathbf{g}})) &= \widehat{\varphi_n}\left(\mathrm{Pf}(\Omega^n_{\mathbf{g}}) \right) \\
                  &= \widehat{\varphi_n}\left(\sum_{j=2}^{2k} (-1)^j x_{i_1 i_j}\mathrm{Pf}(\Omega^n_{\mathbf{g} + e_{i_1} + e_{i_j}}) \right) \text{(using equation \ref{eqn:laplace})}\\
                  &= \sum_{j=2}^{2k} (-1)^j \widehat{\varphi_n}\left(x_{i_1 i_j}\right) \widehat{\varphi_n}\left(\mathrm{Pf}(\Omega^n_{\mathbf{g} + e_{i_1} + e_{i_j}})\right)\\
                  &= \sum_{j=2}^{2k} (-1)^j \frac{\varphi_n(q_{e_{i_1}+ e_{i_j}})}{\varphi_n(q_{\mathbf{0}})}\widehat{\varphi_n}\left(\mathrm{Pf}(\Omega^n_{\mathbf{g} + e_{i_1} + e_{i_j}})\right)
    \end{align*}
    By the induction hypothesis, we have
    \[
    \widehat{\varphi_n}\left(\mathrm{Pf}(\Omega^n_{\mathbf{g} + e_{i_1} + e_{i_j}})\right) = \frac{\varphi_n(q_{\mathbf{g} + e_{i_1} + e_{i_j}})}{\varphi_n(q_{\mathbf{0}})}.
    \]
    Thus, we arrive at the following expression for $\widehat{\varphi_n}(\psi_n(q_{\mathbf{g}}))$:
    \begin{align*}
        \widehat{\varphi_n}(\psi_n(q_{\mathbf{g}})) &= \sum_{j=2}^{2k} (-1)^j  \frac{\varphi_n(q_{e_{i_1} + e_{i_j}})}{\varphi_n(q_{\mathbf{0}})} \frac{\varphi_n ( q_{\mathbf{g} + e_{i_1} + e_{i_j}})}{\varphi_n(q_{\mathbf{0}})} \\
        &= \frac{1}{\varphi_n(q_{\mathbf{0}})^2}\sum_{j=2}^{2k} (-1)^j \varphi_n(q_{e_{i_1} + e_{i_j}} q_{\mathbf{g} + e_{i_1} + e_{i_j}}).
    \end{align*}

    By \cref{lemma: PfaffianExpansion}, the polynomial $F = q_{\mathbf{0}} q_{\mathbf{g}} - \sum_{j=2}^{2k}(-1)^j q_{e_{i_1}+ e_{i_j}} q_{\mathbf{g} + e_{i_1}+ e_{i_j}}$ lies in the kernel of $\varphi_n$; thus, the sum above can be replaced with $\varphi_n(q_{\mathbf{0}} q_{\mathbf{g}})$. Finally, we have that
    \[
        \widehat{\varphi_n}(\psi_n(q_{\mathbf{g}}/q_{\mathbf{0}})) = \frac{\varphi_n(q_{\mathbf{0}}q_{\mathbf{g}})}{\varphi_n(q_{\mathbf{0}})^2} = \frac{\varphi_n(q_{\mathbf{g}})}{\varphi_n(q_{\mathbf{0}})} = (\varphi_n)_0(q_{\mathbf{g}}/q_{\mathbf{0}})
    \]
    completing the induction step and the proof.
\end{proof}

\cref{prop:FactorThroughPfaffian} shows that $U_n = \VV_{\cs_n} \setminus \VV(q_{\mathbf{0}})$ embeds into $\Skew_n(\Cc)$, and in fact, its vanishing ideal in $X_n$ is given by $I_n$. 
As such, we are primarily interested in describing the phylogenetic variety for the $n$-sunlet in terms of $I_n$.
With \cref{thm: fourier coordinates are expected values} and \cref{prop:FactorThroughPfaffian}, we can provide a statistical interpretation of this algebraic result. Namely, that a distribution of states at the leaves of this model is completely determined by its covariance matrix since the $x_{ij}$'s can be interpreted as covariances between variables at the leaves.

Our next result proves that the minors appearing in \cref{thm:main} lie in $I_n$.

\begin{proposition}
\label{thm:gensofkeralpha}
    With $\widehat{\varphi_n}: X_n \to ((A_n)_{\varphi_n(q_{\mathbf{0}})})_0$ and $\Omega^n$ as above, the following polynomials are in $I_n = \ker \widehat{\varphi_n}$
    \begin{enumerate}
        \item [(a)] $\det(\Omega^n_{\{i,j\}, \{k,\ell\}})$ where $1 < i < j < k < \ell \leq n$
        \item [(b)] $\det(\Omega^n_{\{1,i_2,i_3\},\{j_1,j_2,j_3\}})$ where $1 < i_2 < j_1 < j_2 < j_3 < i_3 \leq n$
    \end{enumerate}
    where $\Omega^n_{S, T}$ is the submatrix of $\Omega^n$ whose rows and columns are indexed by $S$ and $T$, respectively.
\end{proposition}

\begin{proof}
The proof uses known results for CFN trees which can be found in \cite{sturmfels2005toric}. 
Let $\varphi_n : Q_n \to A_n$ refer to the parameterization of the $n$-sunlet. We let $\varphi_{\ct_1}$ and $\varphi_{\ct_2}$ be the parameterizations for the subtrees of $\cs_n$ obtained by deleting the reticulation edges $e_{2n}$ and $e_{n+1}$, respectively, and we let $\varphi_{\ct_0}$ be the homogeneous parameterization of the caterpillar subtree $\ct_0$ of $\cs_n$ obtained by deleting all edges incident to the reticulation vertex, the reticulation vertex itself, and leaf 1. Then we have that 
\begin{enumerate}
    \item $\varphi_n = \varphi_{\ct_1} + \varphi_{\ct_2}$, and
    \item $\varphi_n(q_{e_i + e_j}) = (a_0^1 a_0^{n+1} + a_0^1 a_0^{2n}) \varphi_{\ct_0}(q_{e_i + e_j})$ for all $1 < i < j \leq n$, and
    \item $a_0^{2n}\varphi_{\ct_1}(q_{\mathbf{g}}) = a_0^{n+1}\varphi_{\ct_2}(q_{\mathbf{g}})$ whenever the first coordinate of $\mathbf{g}$ is 0.
\end{enumerate}

Recall that the toric ideal $I_{\ct_0}$ associated with the binary tree $\ct_0$ is an iterated toric fiber product of 3-leaf trees \cite{sturmfels2005toric, sullivant2007toric}. This means that for each split $A|B$ of $\ct_0$, the $2 \times 2$ minors of the flattenings $M_g^{A | B}$ are all in the ideal $I_{\ct_0}$ for $g \in \zmodtwo$ \cite{sturmfels2005toric, sullivant2007toric}. Since $1 < i < j < k < \ell$, there is a split $A|B$ of $\ct_0$ where $i,j \in A$ and $k, \ell \in B$. Then the following is a submatrix of $M_1^{A | B}$:
\[
    \omega = \begin{pmatrix}
        q_{e_i + e_k} & q_{e_i + e_\ell} \\
        q_{e_j + e_k} & q_{e_j + e_\ell}
    \end{pmatrix}.
\]
In particular, the determinant of $\omega$ must lie in the kernel of $\varphi_n$; hence, $\det(q_{\mathbf{0}}^{-1} \omega)$ lies in the kernel of $(\varphi_n)_0$. Finally, note that 
    \[
    \psi_n(\det (q_{\mathbf{0}}^{-1}\omega)) = \det (\Omega^n_{\{i,j\},\{k,\ell\}}),
    \] 
and since $\widehat{\varphi_n} \circ \psi_n = (\varphi_n)_0$, we see that $\det (\Omega^n_{\{i,j\},\{k,\ell\}})$ lies in the kernel of $\widehat{\varphi_n}$.

For the $3\times 3$ minors, fix $1 < i_2 < j_1 < j_2 < j_3 < i_3$, and let $\omega$ be the matrix obtained from $\Omega^n_{\{1,i_2,i_3\},\{j_1,j_2,j_3\}}$ by populating it with $q_{e_i + e_j}$'s instead of $x_{i,j}$'s. Consider $\varphi_n(\omega)$ below and rewrite this matrix using points (1) and (3) above.

\[
\begin{pmatrix}
    (\varphi_{\ct_1} + \varphi_{\ct_2})(q_{e_1 + e_{j_1}}) & (\varphi_{\ct_1} + \varphi_{\ct_2})(q_{e_1 + e_{j_2}}) & (\varphi_{\ct_1} + \varphi_{\ct_2})(q_{e_1 + e_{j_3}}) \\
    \left(1 + \frac{a_0^{2n}}{a_0^{n+1}}\right)\varphi_{\ct_1}(q_{e_{i_2} + e_{j_1}}) & \left(1 + \frac{a_0^{2n}}{a_0^{n+1}}\right)\varphi_{\ct_1}(q_{e_{i_2} + e_{j_2}}) & \left(1 + \frac{a_0^{2n}}{a_0^{n+1}}\right)\varphi_{\ct_1}(q_{e_{i_2} + e_{j_3}}) \\
    -\left(1 + \frac{a_0^{n+1}}{a_0^{2n}}\right)\varphi_{\ct_2}(q_{e_{j_1} + e_{i_3}}) & -\left(1 + \frac{a_0^{n+1}}{a_0^{2n}}\right)\varphi_{\ct_2}(q_{e_{j_2} + e_{i_3}}) & -\left(1 + \frac{a_0^{n+1}}{a_0^{2n}}\right)\varphi_{\ct_2}(q_{e_{j_3} + e_{i_3}})
\end{pmatrix}
\]
We claim that the following two matrices are of rank 1. On $\ct_1$, we can choose a split $A|B$ so that $1,i_2 \in A$ and $j_1,j_2,j_3 \in B$, so the following matrix is a submatrix of $M^{A|B}_0$.
\[
\begin{pmatrix}
    q_{e_1 + e_{j_1}} & q_{e_1 + e_{j_2}} & q_{e_1 + e_{j_3}} \\
    q_{e_{i_2} + e_{j_1}} & q_{e_{i_2} + e_{j_2}} & q_{e_{i_2} + e_{j_3}}
\end{pmatrix}
\]
By \cite[Theorem~23]{sturmfels2005toric}, each $2\times 2$ minor lies in $I_{\ct_1}$. Therefore, this matrix has rank 1 after applying $\varphi_{\ct_1}$ to the entries.
Similarly, on $\ct_2$, there is a split $A|B$ with $1,i_3 \in A$ and $j_1,j_2,j_3 \in B$, so the following matrix is a submatrix of $M^{A|B}_1$.
\[
\begin{pmatrix}
    q_{e_1 + e_{j_1}} & q_{e_1 + e_{j_2}} & q_{e_1 + e_{j_3}} \\
    q_{e_{j_1} + e_{i_3}} & q_{e_{j_2} + e_{i_3}} & q_{e_{j_3} + e_{i_3}}
\end{pmatrix}
\]
Therefore, this matrix has rank 1 after applying $\varphi_{\ct_2}$ to the entries.

It follows that we can use row operations to zero out the first row of $\varphi_n(\omega)$; hence, $\varphi_n(\det(\omega))$ is $0$. By localizing this relation, we see that $\det(q_{\mathbf{0}}^{-1}\omega)$ lies in the kernel of $(\varphi_n)_0$; hence, passing this relation through $\psi_n$ yields that
\[
\widehat{\varphi_n}(\det (\Omega^n_{\{1,i_2,i_3\},\{j_1,j_2,j_3\}})) = 0
\] 
as claimed.
\end{proof}


\begin{example}
Consider the ideal $I_6 = \ker(\widehat{\varphi_6})$. This ideal is generated by the following minors of $\Omega^6$: all five $2\times 2$ minors strictly above the diagonal that do not include the first row, which are
\begin{align*}
\mathrm{det}(\Omega_{\{2,3\},\{4,5\}}^6) &= x_{2,4}x_{3,5} - x_{2,5}x_{3,4}, & 
\mathrm{det}(\Omega_{\{2,3\},\{4,6\}}^6) &= x_{2,4}x_{3,6} - x_{2,6}x_{3,4},\\
\mathrm{det}(\Omega_{\{2,3\},\{5,6\}}^6) &= x_{2,5}x_{3,6} - x_{2,6}x_{3,5}, &
\mathrm{det}(\Omega_{\{2,4\},\{5,6\}}^6) &= x_{2,5}x_{4,6} - x_{2,6}x_{4,5}, \\
\mathrm{det}(\Omega_{\{3,4\},\{5,6\}}^6) &= x_{3,5}x_{4,6} - x_{3,6}x_{4,5},
\end{align*}
and the $3\times 3$ minor $\Omega_{\{1,2,6\}, \{3,4,5\}}^6$ highlighted in blue below.
\[
\begin{pmatrix}
0&x_{1,2}&\textcolor{blue}{\boldsymbol{x_{1,3}}}&\textcolor{blue}{\boldsymbol{x_{1,4}}}&\textcolor{blue}{\boldsymbol{x_{1,5}}}&x_{1,6}\\
-x_{1,2}&0&\textcolor{blue}{\boldsymbol{x_{2,3}}}&\textcolor{blue}{\boldsymbol{x_{2,4}}}&\textcolor{blue}{\boldsymbol{x_{2,5}}}&x_{2,6}\\
-x_{1,3}&-x_{2,3}&0&x_{3,4}&x_{3,5}&x_{3,6}\\
-x_{1,4}&-x_{2,4}&-x_{3,4}&0&x_{4,5}&x_{4,6}\\
-x_{1,5}&-x_{2,5}&-x_{3,5}&-x_{4,5}&0&x_{5,6}\\
-x_{1,6}&-x_{2,6}&\textcolor{blue}{\boldsymbol{-x_{3,6}}}&\textcolor{blue}{\boldsymbol{-x_{4,6}}}&\textcolor{blue}{\boldsymbol{-x_{5,6}}}&0
\end{pmatrix}
\]
\end{example}

Let $G_n$ be the set of minors described in the \cref{thm:gensofkeralpha}. 
The previous result establishes that $G_n \subseteq I_n$.  We now show that $G_n$ actually forms a Gr{\"o}bner basis for $I_n$ with respect to the following lexicographic term order $>$ which is given by:
\begin{itemize}
    \item[(a)] $x_{1,i} > x_{j,k}$ for all  $j > 1$ and any  $i, k$. 
    \item[(b)] $x_{1,i} > x_{1,j} \iff i < j$. 
    \item[(c)] $x_{i,j} > x_{k,\ell} \iff i > k $ or $i = k$ and $j < \ell$, for $i, k > 1$. 
\end{itemize}
Under this lexicographic ordering, the initial terms of the minors in $G_n$ are below.
\begin{align*}
    \text{in}_{>}(\det(\Omega^n_{\{i,j\}, \{k,\ell\}})) &= - x_{i,\ell}x_{k,j} \\
    \text{in}_{>}(\det(\Omega^n_{\{1,i_2,i_3\},\{j_1,j_2,j_3\}})) &= -x_{1,j_1}x_{i_2,j_2}x_{j_3,i_3}
\end{align*}

\begin{lemma}\label{lem:gb4-11}
    For the cases $4 \leq n \leq 11$ and the term order given above, $G_n$ forms a Gr{\"o}bner basis for $\langle G_n \rangle$.
\end{lemma}
\begin{proof}
    We prove this by explicit computation in \texttt{Macaulay2} \cite{M2}.  The code for this computation can be found on our MathRepo page \cite{mathrepo}. The computations are done over $\Qq$, which is sufficient to prove $G_n$ is a Gr\"obner basis over $\Cc$.
\end{proof}

\begin{corollary}[The Eleven-to-Infinity Theorem]\label{cor:GB}
For the term order given above and for all $n\geq 6$, $G_n$ forms a Gr{\"o}bner basis.
\end{corollary}
\begin{proof}
    By Buchberger's criterion, it is sufficient to show that on division by $G_n$, the remainder of the $S$-polynomial $S(M_1, M_2)$ is $0$, for all $M_1, M_2 \in G_n$. First, consider $M_1$ and $M_2 \in G_n$ such that both $M_1$ and $M_2$ are $2\times 2$ minors. Then $S(M_1, M_2)$ has remainder 0 on division by $G_n$ since the $2\times 2$ minors of $G_n$ constitute a ladder ideal \cite{Gorla2007}, and under this term order (restricted to monomials $x_{i,j}$ with $i > 1$), it is known that these minors form a Gr{\"o}bner basis.
    
    Next, consider $M_1$ and $M_2 \in G_n$, at least one of which is a $3\times 3$ minor. Let $\mathcal{I} = \{i_1,\ldots, i_m\}$ be the set of indices that appear in the indices of all monomials in $M_1$ and $M_2$ so that $6 \leq m \leq 11$ and consider the subalgebra $\Cc[x_{ij}\ |\ i < j,\ i,j\in\mathcal{I}]$ of $X_n$.

    Let $X_m = \Cc[y_{kl} | 1 \leq k < l \leq m]$. Label the elements of $\mathcal{I}$ so that $1=i_1 < i_2 < \cdots < i_m$. Then $\Cc[x_{ij}\ |\ i < j,\ i,j\in\mathcal{I}]$ is isomorphic to $X_m$ via $\pi(x_{i_k i_l}) = y_{kl}$. It is clear that $\pi(G_n\cap \Cc[x_{ij}\ |\ i < j,\ i,j\in\mathcal{I}]) = G_m$. Since $G_m$ is a Gr{\"o}bner basis by \Cref{lem:gb4-11}, we have that $S(\pi(M_1),\pi(M_2))$ has remainder 0 on division by $G_m$. Observe that $y_{kl} < y_{k'l'}$ if and only if $x_{i_k i_l} < x_{i_{k'}i_{l'}}$ so that $\pi$ preserves initial terms, from which it follows that $S(\pi(M_1),\pi(M_2)) = \pi(S(M_1,M_2))$, and therefore $S(M_1,M_2)$ has remainder 0 on division by $\pi^{-1}(G_m)$. Since $\pi^{-1}(G_m)\subset G_n$, the result follows.
\end{proof}

The previous corollary shows that the set of minors $G_n$ form a Gr\"obner basis for the ideal they generate and \cref{thm:gensofkeralpha} shows that $\langle G_n \rangle \subseteq I_n$. For the rest of this section, our goal is to show that $\langle G_n \rangle = I_n$. Our approach will be based on the following fact.

\begin{lemma}
\label{prop:PrimeDimension} 
Let $I, J \subseteq \Cc[x_1, \ldots, x_n]$ be prime ideals such that $J \subseteq I$ and $\dim(I) = \dim(J)$, then
$I = J$.
\end{lemma}

From this lemma, it is clear that if $\dim(\VV(\langle G_n \rangle)) = \dim((\VV(I_n))$ and $\langle G_n \rangle$ is a prime ideal, then $\langle G_n \rangle = I_n$ and thus $G_n$ forms a Gr\"obner basis for $I_n$. By \cref{prop:sunletDimension}, the dimension of the affine variety $U_n = \VV(I_n)$ is $2n-1$ (since $U_n \subset \VV_{\cs_n}$ is a dense open subset, they must have the same dimension). Therefore, it suffices to show that $\langle G_n \rangle$ is prime and that $\dim(\VV(\langle G_n\rangle)) = 2n-1$. 

\begin{lemma}\label{lemma:radical}
    The ideal $\langle G_n \rangle$ is radical. 
\end{lemma}
\begin{proof}

The leading terms of $G_n$ are squarefree, so the result follows (see \cite[Chapter 4]{coxlittleoshea}).
\end{proof}

\begin{proposition}
The ideal $\langle G_n \rangle$ is prime for $n \geq 4$. 
\end{proposition}

\begin{proof}
    We will show $\langle G_n \rangle$ is prime by showing it is the kernel of a homomorphism between polynomial rings.
    While our ultimate goal is to show that $\langle G_n \rangle = \ker \widehat{\varphi_n}$, the homomorphism we use in this proof is simpler than $\widehat{\varphi_n}$.
    
    The proof is by induction on $n$. Let $S_n = \Cc[s_i,t_j ~|~ 1\leq i,j\leq n]$. Recall that $X_n$ is $\Cc[x_{ij} ~|~ 1 \leq i < j \leq n]$, and consider the ring homomorphism $\beta_n : X_n \to S_n$ defined by
    \[
        x_{ij} \mapsto 
        \begin{cases}
            s_1 t_j + s_j t_1 &\text{if } i = 1 \\
            s_i t_j &\text{if }i > 1.
        \end{cases}.
    \]
    Consider the corresponding morphism $\hat{\beta}_n: \Cc^{2n} \to \Skew_n(\Cc)$. We claim that $\overline{\mathrm{im}(\hat{\beta}_n}) = \VV(G_n)$. It can easily be checked for $n = 4$ and $5$ that equality holds. Indeed, the fact that the kernel of $\beta_n$ is $\langle G_n \rangle$ for $n=4$ and $5$ can be checked immediately using elimination theory and your favorite computer algebra system.
    Moreover, $\VV(G_n) \supseteq \overline{\mathrm{im}(\hat{\beta}_n)}$ since $\beta_n(G_n) = 0$ for all $n$.

    For the reverse inclusion, note that the image of $\hat{\beta}_n$ is a constructible set by Chevalley's Theorem; thus, the Zariski and Euclidean closures of $\mathrm{im}(\hat{\beta}_n)$ actually coincide (see \cite[Corollary 1, Section I.10]{mumfordRedBook}). In fact, for any (non-empty) Zariski open set $\mathcal{U} \subset \Cc^{2n}$, we have the equality $\overline{\hat{\beta}_n(\mathcal{U})} = \overline{\mathrm{im}(\hat{\beta}_n)}$. With this in mind, set 
    \[
        \mathcal{U}_n = \{(s,t) \in \Cc^{n}\times\Cc^{n} ~|~ s_i,t_j \neq 0 \text{ and } \det(\hat{\beta}_n(s,t)_{\{1,2\},\{3,4\}}) \neq 0\}.
    \]
    Suppose $n \geq 6$, $A = (a_{i,j}) \in \VV(G_n)$, and assume that $\VV(G_{n-1}) = \overline{\mathrm{im}(\hat{\beta}_{n-1})}$. Let $\tilde{A}$ be the submatrix of $A$ obtained by deleting the last row and column of $A$. Since $A \in \VV(G_n)$, it is clear that $\tilde{A} \in \VV(G_{n-1})$. Thus, there is a sequence of matrices $\tilde{A}_k \in \hat{\beta}_{n-1}(\mathcal{U}_{n-1}) \subseteq
    \mathrm{im}(\hat{\beta}_{n-1})$ of the form below which converge to $\tilde{A}$:
    \[
        \tilde{A}_k = (a_{i,j}^{(k)}) = \hat{\beta}_{n-1}(s_1^{(k)},\dotsc,s_{n-1}^{(k)}, t_1^{(k)},\dotsc,t_{n-1}^{(k)}) \text{ such that } (s^{(k)},t^{(k)}) \in \mathcal{U}_{n-1}.
    \]
    Our goal is find 2 additional sequences, $s_n^{(k)}$ and $t_n^{(k)}$, so that 
    \[
        \lim_{k \to \infty} \hat{\beta}_n(s_1^{(k)}, \dotsc, s_n^{(k)}, t_1^{(k)}, \dotsc, t_n^{(k)}) = A.
    \]
    We claim that the following to sequences will do the trick.
    \begin{align*}
        t_n^{(k)} &= \frac{a_{2,n}}{s_{2}^{(k)}} \\
        s_n^{(k)} &= \frac{a_{1,n} - s_{1}^{(k)}t_n^{(k)}}{t_{1}^{(k)}}
    \end{align*}
    Note that since $(s^{(k)}_1,\dots,t^{(k)}_{n-1}) \in \mathcal{U}_{n-1}$, these sequences are well-defined. Moreover, by definition of these sequences, we see that 
    \begin{align*}
        \lim_{k\to\infty} s_1^{(k)}t_{n}^{(k)} + s_n^{(k)}t_1^{(k)} &= a_{1,n},\\
        \lim_{k\to\infty} s_2^{(k)}t_{n}^{(k)} &= a_{2,n}.
    \end{align*}
    It remains to check that $\lim_{k\to\infty} s_{i}^{(k)} t_n^{(k)} = a_{i,n}$ for all $3 \leq i \leq n-1$. We consider the case when $3 \leq i \leq n-2$ first. Since $A \in \VV(G_n)$, we know that $a_{i,n} a_{2,n-1} - a_{2,n} a_{i,n-1} = 0$, and as $a_{2,n-1}^{(k)} \to a_{2,n-1}$ and $a_{i,n-1}^{(k)} \to a_{i,n-1}$ as $k \to \infty$, for any $\epsilon > 0$, there exists a $K$ so that for all $k \geq K$, we have that 
    \[
        \left|a_{i,n} - \frac{a_{2,n}a_{i,n-1}^{(k)}}{a_{2,n-1}^{(k)}} \right| < \epsilon.
    \]
    Now, for $k \geq K$, we have the following.
    \begin{align*}
        \left|s_i^{(k)} t_{n}^{(k)} -  a_{i,n}\right| &\leq \left|s_i^{(k)} t_{n}^{(k)} - \frac{a_{2,n}a_{i,n-1}^{(k)}}{a_{2,n-1}^{(k)}}\right| + \left|\frac{a_{2,n}a_{i,n-1}^{(k)}}{a_{2,n-1}^{(k)}} - a_{i,n} \right| \\
        &< \left|s_i^{(k)} t_{n}^{(k)} - \frac{a_{2,n}a_{i,n-1}^{(k)}}{a_{2,n-1}^{(k)}}\right| + \epsilon \\
        &= \left|s_i^{(k)} \frac{a_{2,n}}{s_{2}^{(k)}} - \frac{a_{2,n}a_{i,n-1}^{(k)}}{a_{2,n-1}^{(k)}}\right| + \epsilon \\
        &= |a_{2,n}| \left|\frac{s_i^{(k)}}{s_2^{(k)}} - \frac{s_i^{(k)}}{s_2^{(k)}} \right| + \epsilon \\
        &= \epsilon
    \end{align*}
    Hence, $\lim_{k \to \infty} s_i^{(k)} t_n^{(k)} = a_{i,n}$ for $3 \leq i \leq n-2$.

    Now, we turn our attention to the case when $i = n-1$. There are no $2\times 2$ minors involving $a_{n-1,n}$, so we must use the relation $\det(A_{\{1,2,n\},\{3,4,n-1\}}) = 0$. As in the previous case, we have that for any $\epsilon > 0$, there is a $K$ so that for all $k \geq K$, we have
    \[
    \left|a_{n-1,n} - \frac{a_{1,n-1}^{(k)}a_{2,4}^{(k)}a_{3,n}^{(k)}-a_{1,4}^{(k)}a_{2,n-1}^{(k)}a_{3,n}^{(k)}-a_{1,n-1}^{(k)}a_{2,3}^{(k)}a_{4,n}^{(k)}+a_{1,3}^{(k)}a_{2,n-1}^{(k)}a_{4,n}^{(k)}}{a_{1,3}^{(k)}a_{2,4}^{(k)} - a_{1,4}^{(k)}a_{2,3}^{(k)}} \right| < \epsilon.
    \]
    Note that the denominator never vanishes since $(s_1^{(k)},\dotsc,t_{n-1}^{(k)}) \in \mathcal{U}_{n-1}$. Consider the following when $k \geq K$.
    \begin{align*}
        \left| s_{n-1}^{(k)} t_n^{(k)} - a_{n-1,n} \right| &\leq \left|s_{n-1}^{(k)} t_n^{(k)}- \frac{a_{1,n-1}^{(k)}a_{2,4}^{(k)}a_{3,n}^{(k)}-a_{1,4}^{(k)}a_{2,n-1}^{(k)}a_{3,n}^{(k)}-a_{1,n-1}^{(k)}a_{2,3}^{(k)}a_{4,n}^{(k)}+a_{1,3}^{(k)}a_{2,n-1}^{(k)}a_{4,n}^{(k)}}{a_{1,3}^{(k)}a_{2,4}^{(k)} - a_{1,4}^{(k)}a_{2,3}^{(k)}} \right| \\&+ \left| \frac{a_{1,n-1}^{(k)}a_{2,4}^{(k)}a_{3,n}^{(k)}-a_{1,4}^{(k)}a_{2,n-1}^{(k)}a_{3,n}^{(k)}-a_{1,n-1}^{(k)}a_{2,3}^{(k)}a_{4,n}^{(k)}+a_{1,3}^{(k)}a_{2,n-1}^{(k)}a_{4,n}^{(k)}}{a_{1,3}^{(k)}a_{2,4}^{(k)} - a_{1,4}^{(k)}a_{2,3}^{(k)}} -a_{n-1,n}\right| \\
        &<\left|s_{n-1}^{(k)} t_n^{(k)}- \frac{a_{1,n-1}^{(k)}a_{2,4}^{(k)}a_{3,n}^{(k)}-a_{1,4}^{(k)}a_{2,n-1}^{(k)}a_{3,n}^{(k)}-a_{1,n-1}^{(k)}a_{2,3}^{(k)}a_{4,n}^{(k)}+a_{1,3}^{(k)}a_{2,n-1}^{(k)}a_{4,n}^{(k)}}{a_{1,3}^{(k)}a_{2,4}^{(k)} - a_{1,4}^{(k)}a_{2,3}^{(k)}} \right| + \epsilon
    \end{align*}
    After substituting $a_{i,j}^{(k)} = s_i^{(k)} t_j^{(k)}$ when $1 < i < j$ and $a_{1,j}^{(k)} = s_1^{(k)} t_j^{(k)} + s_j^{(k)} t_1^{(k)}$, we find that the first term in the last line above is identically 0. We see that for $k \geq K$ that $|s_{n-1}^{(k)} t_n^{(k)} - a_{n-1,n}| < \epsilon$; hence, 
    \[
        \lim_{k \to \infty} s_{n-1}^{(k)} t_n^{(k)} = a_{n-1,n}.
    \]

    The above shows that the sequence of matrices
    \[
        A_k = \hat{\beta}_n(s_1^{(k)}, \dotsc, s_{n-1}^{(k)},s_n^{(k)}, t_1^{(k)}, \dotsc, t_{n-1}^{(k)}, t_n^{(k)})
    \]
    converges point-wise to $A$; hence, $A$ lies in the closure of $\mathrm{im}(\hat{\beta}_n)$. This completes the proof that the closure of the image of  $\hat{\beta}_n$ is equal to $\VV(G_n)$.
    
    Since $\VV(G_n) = \overline{\mathrm{im}(\hat{\beta}_n)}$, we see that $\ker \beta_n$ and $\langle G_n \rangle$ have the same radical. Note that $\ker \beta_n$ is radical as it is prime and $\langle G_n \rangle$ is radical by \cref{lemma:radical}. Hence, $\ker \beta_n = \langle G_n \rangle$, completing the proof.
\end{proof}

\begin{proposition}
\label{prop: size of facet}
 For all $n \geq 5$ , $\dim(\langle G_n\rangle) = 2n - 1$. When $n = 4$, $\dim(\langle G_4 \rangle) = 6$.
\end{proposition}

\begin{proof}
When $n = 4$, $\langle G_4 \rangle = 0$, and $\dim(\langle G_n\rangle) = 6$.
When $n = 5$, $\langle G_5 \rangle =\langle x_{2,4}x_{3,5} - x_{2,5}x_{3,4}\rangle$. In particular, the height of $\langle G_5\rangle $ is 1, so the dimension is $9 = 2(5) - 1$ as there are 10 variables.

Now suppose $n \geq 6$.
By \cref{cor:GB} the minors in $G_n$ form a Gr\"obner basis, thus  $\text{in}_{>}(\langle G_n \rangle)$ is generated by the monomials
\begin{align*}
    x_{i,\ell}x_{j,k} &\quad\text{ where } 1 <i <j <k <\ell \leq n \\
    x_{1,j_1}x_{i_2,j_2}x_{j_3,i_3} &\quad\text{ where } 1 < i_2 < j_1 < j_2 < j_3 < i_3 \leq n.
\end{align*}
In particular, we have for the first monomial that $1 < i < \ell-2$ and for the second that $4 \leq j_2 \leq n-2$. Let $S$ be the following subset of generators of $X_n$,
$$ S = \{x_{1,2}, x_{1,n-2}, x_{1, n-1}, x_{1, n}\} \cup \{x_{i, i+1}\,|\,i=2,\ldots,n-1\}\cup\{x_{i, i+2}\,|\, i=2\ldots, n-2\}.$$
We claim that $S$ is independent modulo $\text{in}_{>}(I)$ (i.e. $\mathbb{C}
[S]\cap\text{in}_{>}(I) = \{0\}$). Observe that $x_{1, j_1} \not\in U$ for $3\leq j_1 \leq n-3$ so $x_{1,j_1}x_{i_2,j_2}x_{j_3,i_3} \not\in \mathbb{C}[S]$. Similarly, $x_{i,\ell} \not\in S$ for all $1 < i < \ell-2$ so $x_{i,\ell}x_{j,k} \not\in \mathbb{C}[S]$. Since $\text{in}_{>}(I)$ is monomial, the claim is proven which implies that
\[
\dim(\langle G_n\rangle) \geq |S| = 4 + n-2 + n-3 = 2n - 1. 
\]
Next, we show that $S$ is maximal with respect to the independence property. Since $\langle G_n \rangle$ is prime, all maximally independent subsets modulo $\text{in}_{>}(I)$ have the same cardinality, and this is equal to the dimension of $\mathcal{V}(G_n)$. Thus, to prove the result, it is sufficient to show that $S$ is maximal.
Consider $x_{ij} \not\in S$, so that $i > 1$ and $j > i+2$. Then the $2\times 2$ minor $x_{i,i+2}x_{i+1,j} - x_{i,j}x_{i+1,i+2} \in G_n$ and has initial term $x_{i,j}x_{i+1,i+2} \in \text{in}_{>}(\langle G_n\rangle)$. Since $x_{i+1,i+2} \in S$, we have that $S$ is maximally independent modulo $\text{in}_{>}(\langle G_n\rangle)$.
\end{proof}

By \cref{prop:PrimeDimension}, it now follows that $I_n = \langle G_n\rangle$ proving \cref{thm:main}.

\subsection{Quadratic Generation in Fourier Coordinates}
\label{sec:quadratic generation}

As discussed in \cref{sec:pfaffian}, \cref{prop: size of facet} and \cref{prop:sunletDimension} yield the following two corollaries, which provide complete generating sets for $I_n$ and $J_n$, respectively. 

\begin{corollary}
    The set of $2\times 2$ and $3 \times 3$ minors $G_n$ forms a Gr\"obner basis for $I_n = \ker(\widehat{\varphi_n})$. 
\end{corollary}

\begin{corollary}\label{cor: ideal in fourier coords}
    Let $J_n = \ker((\varphi_n)_0)$ and $I_n = \langle G_n \rangle$. For any minor $m \in G_n$, let $\tilde{m} \in Q_n$ be the polynomial obtained by replacing each $x_{ij}$ with $q_{e_i + e_j}/q_\mathbf{0}$ and let $\widetilde{G_n} = \{\widetilde{m} ~|~ m \in G_n\}$. Then $J_n = \langle \widetilde{G_n} \rangle + \ker(\psi_n)$. 
\end{corollary}
\begin{proof}
Recall that $J_n = \ker((\varphi_n)_0)$ and by \cref{prop:FactorThroughPfaffian} $(\varphi_n)_0 = \widehat{\varphi_n} \circ \psi_n$. Thus, we have that
\[
J_n = \ker((\varphi_n)_0) = \psi_n^{-1}(\ker(\widehat{\varphi_n})) = \psi_n^{-1}(I_n). 
\]
Now observe that $\psi_n$ is essentially the identity on the coordinates $q_{e_i + e_j}/q_\mathbf{0}$ since $\psi_n(q_{e_i + e_j}/q_\mathbf{0}) = x_{ij}$. Thus, we can naturally identify any $m \in G_n$ with the polynomial $\widetilde{m} \in Q_n$ obtained by simply replacing each $x_{ij}$ with $q_{e_i + e_j}/q_\mathbf{0}$ and similarly define $\widetilde{G_n}$ to be the collection of all such $\widetilde{m}$. Now, we claim that $\psi_n^{-1}(I_n) = \langle \widetilde{G_n} \rangle + \ker(\psi_n)$. 
By definition of $\widetilde{G_n}$, we have $\psi_n(\widetilde{G_n}) \subset I_n = \langle G_n\rangle$, so clearly, $\psi_n^{-1}(I_n) \supseteq \langle \widetilde{G_n} \rangle + \ker(\psi_n)$. It remains to show that $\psi_n^{-1}(I_n) \subseteq \langle \widetilde{G_n} \rangle + \ker(\psi_n)$.

Suppose that $f \in \psi_n^{-1}(I_n)$, then $\psi_n(f) \in I_n$ and since $G_n$ generates $I_n$ it holds that $\psi_n(f) = \sum_{m \in G_n} h_m m$ for some polynomials $h_m \in X_n$. Similar to before, let $\widetilde{h_m} \in ((Q_n)_{q_{\mathbf{0}}})_0$ again be the polynomial obtained by replacing each $x_{ij}$ with $q_{e_i + e_j}/q_\mathbf{0}$ and let $\widetilde{f} = \sum_{m\in G_n} \widetilde{h_m}\widetilde{m} \in ((Q_n)_{q_\mathbf{0}})_0$. Observe that $\psi_n(f - \widetilde{f}) = 0$ by construction thus $f - \widetilde{f} \in \ker(\psi_n)$ but $\widetilde{f} \in \langle \widetilde{G_n} \rangle$ thus $f \in \langle \widetilde{G_n} \rangle + \ker(\psi_n)$. 
\end{proof}

With this corollary in mind, we can partially answer Conjecture 6.1 from \cite{cummings2024invariants} which states that in Fourier coordinates, the vanishing ideal for the $n$-sunlet network in $\Pp^{2^{n-1}-1}$ is generated by quadratics. Here we show that the statistically relevant affine open patch (in Fourier coordinates) is generated by quadratics; therefore, up to homogenizing by $q_{\mathbf{0}}$, the vanishing ideal of the projective variety $\cs_n$ is generated by quadratics as well.

\begin{lemma}\label{lem: pfaffian relations}
    The kernel of $\psi_n : Q_n \to X_n$ is generated by quadratics. 
\end{lemma}

\begin{proof}
    We claim that $\ker \psi_n$ is generated by quadratic polynomials of the form
    \[
        f_{\mathbf{g}} = \frac{q_{\mathbf{g}}}{q_\mathbf{0}} - \sum_{j=2}^{2k} (-1)^j \frac{q_{e_{i_1} + e_{i_j}}q_{\mathbf{g} + e_{i_1} + e_{i_j}}}{q_\mathbf{0}^2}
    \]
    where $\mathbf{g} = \sum_{j=1}^{2k} e_{i_j}$ and $k \geq 2$. By \cref{prop: pfaffian laplace expansion}, these relations are all in the kernel of $\psi_n$. 

    Suppose $\psi_n(F) = 0$ for some $F\in Q_n$. 
    For each appearance of $q_\mathbf{g}/q_\mathbf{0}$ (where $\mathbf{g} = \sum_{j=1}^{2k} e_{i_j}$ and $k \geq 2)$ divide $F$ by $f_\mathbf{g}$. This process can be iterated until $F$ is a polynomial in only the $q_{e_i + e_j}/q_{\mathbf{0}}$'s. Let the resulting polynomial be $G$. This has the effect of replacing every such $q_\mathbf{g}/q_\mathbf{0}$ by its Pfaffian expansion in terms the $q_{e_i + e_j}/q_\mathbf{0}$'s. Note that since $\psi_n$ is an isomorphism when restricting to these variables and $\psi_n(G) = 0$, it follows that $G = 0$. In particular, $F \in \langle f_\mathbf{g} ~|~ \mathbf{g} = \sum_{j=1}^{2k} e_{i_j},\, k\geq 2\rangle$.
\end{proof}

\begin{proposition}
    $J_n$ is generated by quadratics.
\end{proposition}

\begin{proof}

    For $1 < i_2 < j_1 < j_2 < j_3 < i_3 \leq n$, set $f_{\{1,i_2,i_3\},\{j_1,j_2,j_3\}}$ to be
    \[
        \frac{q_{e_1 + e_{j_1}} q_{e_{i_2} + e_{j_2} + e_{j_3} + e_{i_3}} -
        q_{e_1 + e_{j_2}} q_{e_{i_2} + e_{j_1} + e_{j_3} + e_{i_3}} + 
        q_{e_1 + e_{j_3}} q_{e_{i_2} + e_{j_1} + e_{j_2} + e_{i_3}} - 
        q_{e_{i_2} + e_{i_3}} q_{e_1 + e_{j_1} + e_{j_2} + e_{j_3}}}{q_\mathbf{0}^2}.
    \]
    We claim that $-\psi_n(f_{\{1,i_2,i_3\},\{j_1,j_2,j_3\}}) = \det(\Omega_{\{1,i_2,i_3\},\{j_1,j_2,j_3\}}^n)$. Indeed, we have the following.
    \begin{align*}
        -\psi_n\left(\frac{q_{e_1 + e_{j_1}} q_{e_{i_2} + e_{j_2} + e_{j_3} + e_{i_3}}}{q_\mathbf{0}^2}\right) &= -x_{1j_1}x_{i_2i_3}x_{j_2j_3}+x_{1j_1}x_{i_2j_3}x_{j_2i_3}-x_{1j_1}x_{i_2j_2}x_{j_3i_3} \\
        \psi_n\left(\frac{q_{e_1 + e_{j_2}} q_{e_{i_2} + e_{j_1} + e_{j_3} + e_{i_3}}}{q_\mathbf{0}^2}\right) &= x_{1j_2}x_{i_2i_3}x_{j_1j_3}-x_{1j_2}x_{i_2j_3}x_{j_1i_3}+x_{1j_2}x_{i_2j_1}x_{j_3i_3}\\
        -\psi_n\left(\frac{q_{e_1 + e_{j_3}} q_{e_{i_2} + e_{j_1} + e_{j_2} + e_{i_3}}}{q_\mathbf{0}^2}\right) &= -x_{1j_3}x_{i_2i_3}x_{j_1j_2}+x_{1j_3}x_{i_2j_2}x_{j_1i_3}-x_{1j_3}x_{i_2j_1}x_{j_2i_3}\\
        \psi_n\left(\frac{q_{e_{i_2} + e_{i_3}} q_{e_1 + e_{j_1} + e_{j_2} + e_{j_3}}}{q_\mathbf{0}^2}\right) &= x_{1j_3}x_{i_2i_3}x_{j_1j_2}-x_{1j_2}x_{i_2i_3}x_{j_1j_3}+x_{1j_1}x_{i_2i_3}x_{j_2j_3}
    \end{align*}
    Note that if we take the sum of the expressions on the right, the first term in each of the first three polynomials cancels with a term in the last expression. The resulting polynomial is exactly $\det(\Omega_{\{1,i_2,i_3\},\{j_1,j_2,j_3\}}^n)$.
    
    For $1 < i < j < k < \ell\leq n$, set $f_{\{i,j\},\{k,\ell\}}$ to be
    \[
        \frac{q_{e_i + e_k} q_{e_j + e_\ell} - q_{e_i + e_\ell} q_{e_j + e_k}}{q_\mathbf{0}^2}.
    \]
    By inspection, we see $ \psi_n(f_{\{i,j\},\{k,\ell\}}) = \det(\Omega_{\{i,j\},\{k,\ell\}}^n)$. Together with \cref{lem: pfaffian relations}, this shows that $J_n$ from \cref{cor: ideal in fourier coords} is generated by quadratics in the variables $q_\mathbf{g}/q_\mathbf{0}$.
\end{proof}

\section{Inferring Phylogenetic Networks in Pfaffian Coordinates}
\label{sec:inference}

 In this section, we demonstrate how to use our results to infer leaf-labeled sunlet networks from aligned DNA sequence data. First, however, we begin with an identifiability result for sunlet networks using the Gr\"obner basis described in the previous section.

\subsection{Identifiability of Sunlets}
\begin{proposition}
\label{prop:ident}
    For $n \geq 6$ let $\mathcal{N}_1$ and $\mathcal{N}_2$ be two distinct $n$-sunlet networks under the CFN model, with corresponding phylogenetic network varieties $V_{\mathcal{N}_1}$ and $V_{\mathcal{N}_2}$. Then $V_{\mathcal{N}_1} \not\subset V_{\mathcal{N}_2}$  and $V_{\mathcal{N}_2} \not\subset V_{\mathcal{N}_1}$.
\end{proposition}
\begin{proof}
    It is sufficient to prove the result on the affine open patch where $q_{\mathbf{0}} \neq 0$. For $i = 1,2$, let $I_i$ be the kernel of the corresponding map $\alpha_i$. Since $\mathcal{N}_1$ and $\mathcal{N}_2$ have the same underlying mixed graph and differ only by a permutation of the leaf labels, the elements of $I_i$ are given by the minors described in the previous section, with permuted rows and columns. (i.e. if $\theta\in S_n$ is a permutation of the $n$ leaves then $\theta(\det(M_{\{i,j\}\{k,l\}})) = \det(M_{\{\theta(i),\theta(j)\}\{\theta(k),\theta(l)\}}$).

    Let us assume that $\mathcal{N}_1$ is the $n$-sunlet with leaf $1$ below the reticulation and leaves $2,\ldots,n$ ordered clockwise (as in \cref{fig:4SunletAndTrees}(A)). Then $I_1$ has Gr\"obner basis $G_n$. Let $\theta \in S_n$ be the permutation which, when applied to the leaves of $\mathcal{N}_1$, gives $\mathcal{N}_2$. Then for some $i < j$ we must have $\theta(i)>\theta(j)$. If $i=1$ then $\det(M_{\{1,2,6\}\{3,4,5\}}) \in I_1$ so $\theta(\det(M_{\{1,2,6\}\{3,4,5\}})) = \det(M_{\{\theta(1),\theta(2),\theta(6)\}\{\theta(3),\theta(4),\theta(5)\}}) \in I_2$, but then $\theta(\det(M_{\{1,2,6\}\{3,4,5\}})) \not\in I_1$ since the first row of this minor is not indexed by $1$. If $i > 1$, suppose that $j < n-1$, and consider $\det(M_{\{i,j\}\{n-1,n\}})$. Then $\theta(\det(M_{\{i,j\}\{n-1,n\}}) \in I_2 \setminus I_1$. We leave the remaining cases (i.e., when $j \geq n-1$) to the reader. It follows that $I_2 \not\subset I_1$. All other cases can be obtained by symmetry.
\end{proof}

The result above tells us that for a fixed $n\geq 6$, distinct $n$-sunlets are generically identifiable from the marginal distribution of states at the leaves of the sunlet. In particular, we can determine the reticulation vertex of the sunlet network from this distribution. For $n\leq 4$ we do not have generic identifiability of $n$-sunlets, because each $n$-sunlet variety has the same dimension as the space it lives in, and thus fills up the whole space {\cite{gross2023dimensions}}.
For $n = 5$, the 5-sunlets are also not generically identifiable. Indeed, taking $\cs_5$ and swapping the leaves 4 and 5 or leaves 2 and 3 yields the same ideal.
We identify the states of the CFN model with purines (nucleotides A and G) and pyrimidines (nucleotides C and T). In the remainder of this section, our aim is to evaluate the minors of $G_n$ at a matrix formed from the empirical distributions obtained from sequence data for each possible leaf-labeling of the $n$-sunlet. Proposition \ref{prop:ident} tells us that for the true leaf-labeling, all of the minors should evaluate to 0 (although accounting for error in the empirical distribution, we expect it to lie very close to 0), and for all other leaf-labelings, the minors should not all evaluate to zero.

\subsection{Implementation and Results on Simulated Data}

We simulated multiple sequence alignments of DNA sequence data from sunlets for $n = 6,7,$ and $8$ under the Kimura 3-Parameter (K3P) evolutionary model {\cite{kimura1981estimation}}. Here, we think of this model as being the general group-based model for the Klein-4 group $\zmodtwo\times\zmodtwo$ (see e.g. {\cite{casanellas2007geometry}} for details of the K3P model in the algebraic geometry setting). For each $n$, we simulated 1,000 alignments of lengths 1,000bp, 10,000bp, and 100,000bp. Note that a multiple-sequence alignment of DNA data can be converted to a multiple-sequence alignment of purine/pyrimidine data under the usual RY-coding. When identifying the state space $\{\rm{A,G,C,T}\}$ with 
$\mathbb{Z}/2\mathbb{Z}\times\mathbb{Z}/2\mathbb{Z}$ (as is the case for the K3P model), we can think of this conversion as a projection from $\mathbb{Z}/2\mathbb{Z}\times\mathbb{Z}/2\mathbb{Z}$ onto $\mathbb{Z}/2\mathbb{Z}$. If our DNA sequence data is generated under the K3P model on a network $\mathcal{N}$, when we apply this projection to the Fourier coordinates, we are left with coordinates satisfying the CFN relations on the same network $\mathcal{N}$ (see the proof of \cite[Theorem~4.7]{hollering2021identifiability} for further details). We, therefore, have a multiple sequence alignment generated under the CFN model on the network $\mathcal{N}$. 

\begin{algorithm}
\caption{High-level overview of our algorithm to infer a sunlet network from aligned DNA sequence data. The sunlet with the lowest score is taken to be the most likely one that generated the data.}\label{alg:infer}
\KwData{Multiple DNA sequence alignment $A$  from $n$ species.}
\KwResult{Score for every $n$-sunlet}
Convert $A$ to a purine/pyrimidine alignment $A'$; \\
Calculate the empirical distribution $p$ of leaf patterns from $A'$; \\
\For{each $n$-sunlet $\mathcal{N}$}{
    $p_{\mathcal{N}} \gets p$ with indices permuted to align with $\mathcal{N}$;\\
    \For{each pair $(i,j)$ with $i < j$}{
        $q_{e_i + e_j} \gets e_i + e_j$ coordinate of Fourier transform of $p_{\mathcal{N}}$;
    }
    $\Omega_{\mathcal{N}} \gets$ matrix with $(i,j)$-entry given by $q_{e_i + e_j}$; \\
    ${\rm Scores}[\mathcal{N}] \gets 0$; \\
    \For{each minor $m \in G_n$}{
        ${\rm Scores}[\mathcal{N}] \gets {\rm Scores}[\mathcal{N}] + |m(\Omega_{\mathcal{N}})|$;
    }
  }
  Sort ${\rm Scores}$ in ascending order; \\
  \For{each $n$-sunlet $\mathcal{N}$}{
    Print $\mathcal{N}$, $\rm{Scores}[\mathcal{N}]$;
  }
\end{algorithm}

Our algorithm to infer a leaf-labeled $n$-sunlet from aligned DNA sequence data is as follows (see also Algorithm \ref{alg:infer}). First, an input of multiple DNA sequence alignment of $n$ species is converted to a multiple sequence alignment of purine/pyrimidine data. From this, an empirical distribution of observed leaf patterns is calculated for every possible leaf-labeling of the $n$-sunlet (up to network equivalence) by permuting the sequences in the alignment. Each distribution is then transformed via the discrete Fourier transform, and the matrix $\Omega^n$ is constructed. Each minor from Corollary \ref{cor: ideal in fourier coords} is then calculated, and the absolute values are summed to give a score for the leaf-labeling. Each leaf-labeling is then ranked by this score from lowest to highest, and the labeling with the lowest score is chosen as the \lq correct\rq\ labeling of the sunlet. Python scripts that implement this and simulate aligned sequence data are available at \cite{mathrepo}, and use the Python packages \texttt{scikit-bio} \cite{scikitbio} and \texttt{mpmath} \cite{mpmath}.

Tables \ref{tab:n6}, \ref{tab:n7}, and \ref{tab:n8} give the distribution of the ranking of the true network for the three alignment lengths we simulated. We observe that as sequence lengths approach $100,000$bp, the percent of datasets where the true network was identified approaches $100$\%.

\begin{table}[h!]
 \begin{center}
 \caption{\% Rankings of the true network for $n=6$ over 1,000 datasets. For each dataset, we recorded the ranking of the true network that generated the data after applying Algorithm \ref{alg:infer}. A ranking of 1 means that our algorithm correctly identifies the true network. For $n=6$, there are a total of 360 sunlet networks to test.}
  \label{tab:n6}
 \begin{tabular}{c|ccccccccccccc}
                & \multicolumn{10}{|c}{Rank of True Network}\\
	  Alignment Length & 1 & 2 & 3 & 4 & 5 & 6 & 7 & 8 & 9 & $\geq$ 10 \\ 
   \hline
	  1,000 & 64.0\% & 13.2\% & 4.5\% & 2.6\% & 5.3\% & 4.1\% & 0.5\% & 0.4\% & 1.3\% & 4.1 \% \\
	  10,000 & 95.4\% & 3.6\% & 0.2\% & 0\% & 0.6\% & 0.1\% & 0\% & 0\% & 0\% & 0.1\%  \\
	  100,000 & 99.5\% & 0.5\% & 0\% & 0\% & 0\% & 0\% & 0\% & 0\% & 0\% & 0\% 
	 \end{tabular} 
\end{center}
\end{table}

\begin{table}[h!]
 \begin{center}
 \caption{Rankings of the true network for $n=7$ over 1,000 datasets. For $n=7$, there are a total of 2520 sunlet networks to test.}
  \label{tab:n7}
 \begin{tabular}{c|ccccccccccccc}
            & \multicolumn{10}{|c}{Rank of True Network}\\
	  Alignment Length & 1 & 2 & 3 & 4 & 5 & 6 & 7 & 8 & 9 & $\geq$ 10 \\ \hline
        1,000 & 71.1\% & 17.5\% & 1.7\% & 1.9\% & 4.1\% & 1.1\% & 0\% & 0.1\% & 0.8\% & 1.7\% \\
	  10,000 & 96.2\% & 3.5\% & 0\% & 0.1\% & 0.2\% & 0\% & 0\% & 0\% & 0\% & 0\% \\
	  100,000 & 99.6\% & 0.3\% & 0\% & 0\% & 0.1\% & 0\% & 0\% & 0\% & 0\% & 0\% 
	 \end{tabular} 
\end{center}
\end{table}

\begin{table}[h!]
 \begin{center}
 \caption{Rankings of the true network for $n=8$ over 1,000 datasets. For $n=8$, there are a total of 20,160 sunlet networks to test.}
  \label{tab:n8}
 \begin{tabular}{c|ccccccccccccc}
             & \multicolumn{10}{|c}{Rank of True Network}\\
	  Alignment Length & 1 & 2 & 3 & 4 & 5 & 6 & 7 & 8 & 9 & $\geq$ 10 \\ \hline
        1,000 & 70.6\% & 18.7\% & 1.3\% & 1.7\% & 4.6\% & 0.6\% & 0.1\% & 0.2\% & 1.1\% & 1.1\% \\
	  10,000 & 96.4\% & 3.2\% & 0\% & 0\% & 0.3\% & 0.1\% & 0\% & 0\% & 0\% & 0\% \\
	  100,000 & 99.6\% & 0.4\% & 0\% & 0\% & 0\% & 0\% & 0\% & 0\% & 0\% & 0\% 
	 \end{tabular} 
\end{center}
\end{table}

Our implementation of the algorithm is parallelized so that the score for each sunlet network can be evaluated simultaneously. \cref{fig:times}.a shows the times taken for each dataset, where each dataset was analyzed with 4 CPUs so that at most 4 leaf-labelings were evaluated simultaneously. Here, in the worst case, the running time for alignments of length 100kbp for $n=8$ take under 1 hour. Using more CPUs and making further optimizations to our algorithm could decrease this further. The majority ($>50\%$) of the running time is spent performing the Fourier transformation of the leaf patterns and constructing the matrix $\Omega$.  Only a small proportion of time is spent on evaluating invariants (see \cref{fig:times}.b).

\begin{figure}[h!]
    \centering
    \includegraphics[scale=0.65]{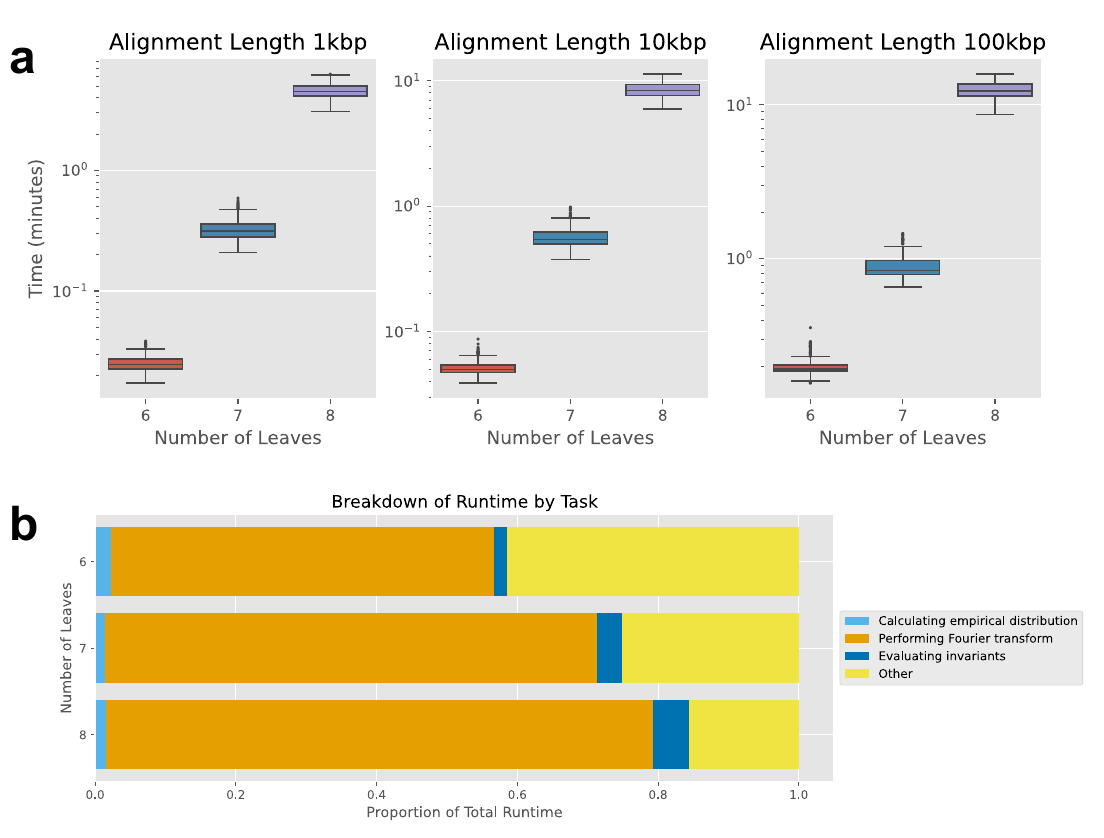}%
    \caption{{\bf a.} Box plots displaying the run times for $n=6,7,$ and $8$, split by alignment length. {\bf b.} Breakdown of run time by task over 100 alignments of length 1kbp for each $n=6,7,8$.}
    \label{fig:times}
\end{figure}%

\section{Discussion}
\label{sec:5}

We have given an explicit description of the vanishing ideal of the statistically relevant affine open patch of the variety associated to a sunlet network under the CFN substitution model. We showed that the number of coordinates needed to parameterize the model can be reduced from $2^{n-1}$ to $\binom{n}{2}$, and we have given a Gr\"{o}bner basis of the ideal in terms of certain minors of a skew-symmetric matrix $\Omega$. The entries $x_{ij}$ of $\Omega$ correspond to the Fourier coordinates $q_{e_i + e_j}$. By \cref{thm: fourier coordinates are expected values}, these Fourier coordinates are equal to $\mathbb{E}[\chi(X_i) \chi(X_j)]$. Since the odd moments are all 0, we can interpret the Fourier coordinates as \emph{covariances}. Furthermore, for any $i$, we have that $\mathrm{Var}(\chi(X_i)) = 1$ since $\chi(X_i)^2 = 1$. This leads us to a reinterpretation of \cref{thm:main}.

\begin{theorem}
    Let $U_n = \VV_{\cs_n} \setminus \VV(q_{\mathbf{0}})$ be the statistically relevant affine open patch of the CFN sunlet model. Let $\ct_0$ be the caterpillar subtree of $\cs_n$ obtained by deleting all edges incident to the reticulation vertex, the reticulation vertex itself, and leaf 1. Then:
    \begin{enumerate}
        \item [(a)] The map which sends a probability distribution $p$ to its covariance matrix $\Sigma$ is injective.
        \item [(b)] After passing to the new coordinate system, $\Sigma$ lies in the phylogenetic variety if and only if the following rank conditions hold.
        \begin{enumerate}
            \item [(i)] For every split $A|B$ of $\ct_0$, $\mathrm{rk}(\Sigma_{A,B}) \leq 1$. Note that $1 \notin A \cup B$.
            \item [(ii)] In $\cs_n$, if we delete any two non-reticulation cycle edges, we arrive at a bipartition of the leaves of $\cs_n$, which we denote as $A|B$. Note that one side of this partition contains $\{1,2,n\}$. For all such bipartitions, we have $\mathrm{rk}(\Sigma_{A,B}) \leq 2$.
        \end{enumerate}
    \end{enumerate}
\end{theorem}

\begin{proof}
    The first statement follows since all the variances are 1, so no information is gained or lost by replacing $\Omega$ with $\Sigma$. For the second statement, note that $G_n$ still defines a generating set for the vanishing ideal of these covariance matrices as the minors appearing in $G_n$ will only differ by $\pm 1$ after replacing $\Sigma$ with $\Omega$. The rank 1 conditions correspond to the $2\times 2$ minors of $G_n$, and the rank 2 conditions correspond to the vanishing of the $3\times 3$ minors in $G_n$.
\end{proof}

We mention this here since an analogous result for trees has implicitly appeared in the literature (e.g., \cite{Speyer-Sturmfels, Sturmfels2020}). In fact, writing down an analogous statement for trees is almost immediate in light of \cite{sturmfels2005toric}. When looking at the statistically relevant affine open patch for a tree $\ct$, the model is defined by $\mathrm{rk}(\Sigma_{A, B}) \leq 1$ for all splits $A|B$ of $\ct$. Therefore, it seems plausible that higher level CFN networks may also be determined by rank constraints on submatrices of $\Sigma$. As evidence, we present one level-2 network where this holds. 

\begin{figure}

    \centering
    \begin{subfigure}[c]{0.4\linewidth}
        \centering
    
            \begin{tikzpicture}[scale = 0.5]
                \newcommand*\pointsuarjz{491.2444875527841/549.4754611945466/0/1,505.66466483556593/493.35854732048256/1/,404.3896991966817/466.56233695652935/2/2,457.17407230876313/448.50521789938784/3/,429.6442869270844/350.6437426072139/4/3,477.787916913392/380.6225305792318/5/,552.7645863012605/224.210937052743/6/4,585.2895416686697/270.10353401834186/7/,667.3268753640284/191.99913815464515/8/5,652.5457406957456/246.0605865626175/9/,753.4158012619172/277.9520189019479/10/6,698.7806699437413/293.6054563545368/11/,726.5512593404143/387.2655543619151/12/7,678.1627669721956/361.0762418336222/13/,601.5817453787387/519.2655912770948/14/8,572.6070583975231/471.9167038473043/15/,550.0262732142388/343.50790274908866/16/,606.6570986151664/398.97400523199303/17/}
                \newcommand*\edgesuarjz{1/0,3/2,5/3,5/4,7/6,9/7,9/8,11/9,11/10,13/12,15/14,16/5,16/7,17/15,17/16}
                
                \newcommand*\scaleuarjz{0.02}
                
                \foreach \x/\y/\z/\w in \pointsuarjz {
                    \node (\z) at (\scaleuarjz*\x,-\scaleuarjz*\y) [circle,draw,inner sep=0pt] {$\w$};
                }        
                
                \foreach \x/\y in \edgesuarjz {
                    \draw (\x) -- (\y);
                }
                
                \draw[dashed, ->] (11) -- (13);
                \draw[dashed, ->] (17) -- (13);
                \draw[dashed, ->] (15) -- (1);
                \draw[dashed, ->] (3) -- (1);
            \end{tikzpicture}
        \end{subfigure}
        \begin{subfigure}[c]{0.4\linewidth}
            \centering 
                \resizebox{0.6\textwidth}{!}{%
                    \begin{tabular}{|c|c|c|}
                    \hline
                    Rows ($A$) & Columns ($B$) & Rank ($d$) \\ \hline
                    $\{2, 3\}$ & $\{4, 5, 6, 7, 8\}$& 1 \\ \hline
                    $\{2, 3, 4\}$& $\{5, 6, 7\}$& 1 \\ \hline
                    $\{2, 3, 4, 5\}$& $\{6, 7\}$& 1 \\ \hline
                    $\{2, 4, 5\}$& $\{3, 6, 7\}$& 2 \\ \hline
                    $\{2, 6, 7\}$& $\{3, 4, 5\}$& 2 \\ \hline
                    $\{1, 2, 3, 7, 8\}$& $\{4, 5, 6\}$& 2 \\ \hline
                    $\{1, 2, 3, 4, 8\}$& $\{5, 6, 7\}$& 2 \\ \hline
                    $\{1, 2, 3, 8\}$& $\{4, 5, 6, 7\}$& 2 \\ \hline
                    $\{1, 2, 8\}$& $\{3, 4, 5, 6, 7\}$& 2 \\ \hline
                    \end{tabular}%
            }
        \end{subfigure}
            \caption{On the left, a level-2 network with 8 leaves. On the right, a table listing rank constraints of the form $\mathrm{rk}(\Sigma_{A,B}) \leq d$.}
            \label{fig: level2example}
\end{figure}

\begin{example}
\label{ex: a level-2 example}
    The network in \cref{fig: level2example} is a level-2 network. It cannot be decomposed into sunlets and trees using the toric fiber product, and thus, does not fall into the paradigm set out in this paper. However, as we shall show below, the vanishing ideal is determined by rank constraints on the covariance matrix as in the level-1 case. 
    
    Let $I$ be the vanishing ideal for this network. First of all, one can show that for this network (and in fact all networks) that the ideal is homogeneous with respect to the multigrading defined at the end of \cref{subsec: phylogenetic invariants}. Then we would like to repeat the proof of \cref{prop:FactorThroughPfaffian} to show that the model is determined completely by covariances. As stated, this doesn't work, the Pfaffian technique fails. However, we can still try to find polynomials of a form close enough to those appearing in \cref{lemma: PfaffianExpansion} to replicate the proof of \cref{prop:FactorThroughPfaffian}. Explicitly, for each $\mathbf{g} = \sum_{j=1}^{2k} e_{i_j}$ with $k \geq 2$, we want to find homogeneous polynomials of multi-degree $(2,\mathbf{g})$ of the form below.
    \[
        q_{\mathbf{0}} q_{\mathbf{g}} + \sum_{\mathbf{h},\mathbf{k}} c_{\mathbf{h},\mathbf{k}} q_{\mathbf{h}} q_{\mathbf{k}}
    \]
    For example,
    \[q_{00000000}q_{11111111}-q_{00001100}q_{11110011}+q_{00001010}q_{11110101}-q_{00000110}q_{11111001}\]
    lies in $I$ and thus, after setting $q_{\mathbf{0}} = 1$, we can solve for $q_{11111111}$ in terms of the lower order moments. If a polynomial of this form occurs for all appropriate $\mathbf{g}$, then we can recursively find formulas for $q_{\mathbf{g}}$ in terms of the covariances $q_{e_i + e_j}$.
    
    As it turns out, one can find a polynomial of this form for every $\mathbf{g} = \sum_{j=1}^{2k} e_{i_j}$ with $k \geq 2$ for this level-2 network. 
    This computation is not too intensive since such a polynomial must be of degree $(2,\mathbf{g})$, reducing the search space. Moreover, \texttt{MultigradedImplicitization} \cite{cummings2023computing}, a Macaulay2 package, is designed to solve such problems. It finds all the desired polynomials in under a minute (including the one listed above).
    
    Once we have these polynomials in hand, we can adapt the proof of \cref{prop:FactorThroughPfaffian} to show that the model is determined by its covariance matrix $\Sigma$. After reducing the parameterization from $128$ variables to $28$, we can again use \texttt{MultigradedImplicitization} to find all generators of the ideal up to degree $4$. All the invariants found were minors of $\Sigma$, and they exactly define rank constraints which can be read off the table in \cref{fig: level2example}. Surprisingly, there are no minimal degree 4 generators of $I$.

    These computations do not guarantee that these rank constraints generate the entire kernel. To find a genuine certificate that these are a generating set, one would need to solve the implicitization problem using Gr\"obner bases. This is infeasible in this case, as there are 28 variables which are parameterized by $38$ parameters, and the expressions defining the $q_{\mathbf{g}}$'s are polynomials with 4 terms of degree 19.
    
    That being said, these rank constraints listed in \cref{fig: level2example} are almost certainly correct. The rank of the Jacobian of the parameterization of the model is 17. Macaulay2 fails to solve the whole implicitization problem; however, it is able to show that the ideal generated by these rank constraints defines a dimension 17 variety and is prime over $\Qq$. While not a proof, it is certainly good evidence that we have found a generating set for $I$.
\end{example}

Methods for inferring quartet trees directly from topology invariants were proposed in \cite{casanellas2007performance}, and more recently for small phylogenetic networks in \cite{barton2022statistical} and \cite{martin2023algebraic}. In both cases for phylogenetic networks, calculating a Gr\"obner basis from which to choose the invariants is a significant hurdle, and indeed, this limited \cite{martin2023algebraic} to looking at 4-leaf networks under the JC and K2P evolutionary models. Here, we have explicitly described the invariants of $n$-sunlets under the CFN model for all $n\geq 6$, thereby removing this difficulty for these cases. Whilst the set of invariants we give is not a full generating set for the whole projective variety, it is a full generating set for an affine open patch of the model that contains the \lq stochastic\rq\ part of the model and is therefore sufficient for use on real data. Furthermore, Proposition \ref{prop:ident} tells us that our method is statistically consistent, and in our results, we observe a convergence rate higher than that observed in \cite{martin2023algebraic}. This may be because we are \lq cutting down\rq\ the variety and only considering an affine open patch.

For each $n$-sunlet, we have $\frac{n!}{2}$ possible leaf-labelings. At the same time, the number of minors we evaluate for each leaf-labeling grows with each $n$, so we expect the running time to grow with each increase in $n$. Nonetheless, the growth of the number of minors is polynomial in $n$, which is a significant improvement on the exponential growth in the number of generators for the whole sunlet ideal. Furthermore, the number of variables in the ring $Q_n$, where the sunlet ideal lives, also grows exponentially, whereas the number of variables in the ring $X_n$ grows polynomially. Thus, by considering only the affine open patch $\mathcal{V}_{\cs_n}$, we can significantly decrease the growth rate against the full sunlet variety.

\section*{Acknowledgements}
SM was supported by the Biotechnology and Biological Sciences Research Council, part of UK Research and Innovation, through the Core Capability Grant BB/CCG1720/1 at the Earlham Institute and is grateful for HPC support from NBI’s Research Computing group. 
BH was supported by the Alexander von Humboldt Foundation.  EG and IN were supported by the National Science Foundation grant DMS-1945584. BH and SM are grateful for funding from the Earlham Institute's Flexible Talent Mobility Award (BB/X017761/1), which facilitated a research exchange on which part of this work was performed. 

This project was initiated at the ``Algebra of Phylogenetic Networks Workshop" held at the University of Hawai`i at M\={a}noa and supported by National Science Foundation grant DMS-1945584. Additional parts of this research were performed while JC, EG, BH, and IN were visiting the Institute for Mathematical and Statistical Innovation (IMSI) for the semester-long program on ``Algebraic Statistics and Our Changing World." IMSI is supported by the National Science Foundation (Grant No. DMS-1929348).

We would also like to thank David Speyer for suggesting the connection to Pfaffians.

\bibliographystyle{plain}
\bibliography{refs.bib}

\begin{thebibliography}{10}

\bibitem{allman2019nanuq}
Elizabeth~S Allman, Hector Ba{\~n}os, and John~A Rhodes.
\newblock Nanuq: a method for inferring species networks from gene trees under the coalescent model.
\newblock {\em Algorithms for Molecular Biology}, 14:1--25, 2019.

\bibitem{allman2019species}
Elizabeth~S. Allman, Colby Long, and John~A. Rhodes.
\newblock Species tree inference from genomic sequences using the log-det distance.
\newblock {\em SIAM Journal on Applied Algebra and Geometry}, 3(1):107--127, 2019.

\bibitem{allman2003invariants}
Elizabeth~S. Allman and John~A. Rhodes.
\newblock Phylogenetic invariants for the general markov model of sequence mutation.
\newblock {\em Mathematical Biosciences}, 186(2):113--144, 2003.

\bibitem{allman2008gmm}
Elizabeth~S. Allman and John~A. Rhodes.
\newblock Phylogenetic ideals and varieties for the general markov model.
\newblock {\em Advances in Applied Mathematics}, 40(2):127--148, 2008.

\bibitem{banos2019identifying}
Hector Ba{\~n}os.
\newblock Identifying species network features from gene tree quartets under the coalescent model.
\newblock {\em Bulletin of mathematical biology}, 81:494--534, 2019.

\bibitem{barton2022statistical}
Travis Barton, Elizabeth Gross, Colby Long, and Joseph Rusinko.
\newblock Statistical learning with phylogenetic network invariants.
\newblock {\em arXiv 2211.11919}, 2022.

\bibitem{casanellas2007geometry}
Marta Casanellas and Jes{\'u}s Fern{\'a}ndez-S{\'a}nchez.
\newblock Geometry of the kimura 3-parameter model.
\newblock {\em Advances in Applied Mathematics}, 41:265--292, 2007.

\bibitem{casanellas2007performance}
Marta Casanellas and Jes{\'u}s Fern{\'a}ndez-S{\'a}nchez.
\newblock Performance of a new invariants method on homogeneous and nonhomogeneous quartet trees.
\newblock {\em Molecular Biology and Evolution}, 24(1):288--293, 2007.

\bibitem{casanellas2021rank}
Marta Casanellas and Jes{\'u}s Fern{\'a}ndez-S{\'a}nchez.
\newblock Rank conditions on phylogenetic networks.
\newblock In {\em Extended Abstracts GEOMVAP 2019: Geometry, Topology, Algebra, and Applications; Women in Geometry and Topology}, pages 65--69. Springer, 2021.

\bibitem{cavender1978taxonomy}
James~A Cavender.
\newblock Taxonomy with confidence.
\newblock {\em Mathematical Biosciences}, 40(3-4):271--280, 1978.

\bibitem{Cayley1849}
A.~Cayley.
\newblock Sur les déterminants gauches. (suite du mémoire t. xxxii. p. 119).
\newblock {\em Journal für die reine und angewandte Mathematik}, 38:93--96, 1849.

\bibitem{coxlittleoshea}
David~A. Cox, John Little, and Donal O'Shea.
\newblock {\em Ideals, Varieties, and Algorithms: An Introduction to Computational Algebraic Geometry and Commutative Algebra}.
\newblock Springer Publishing Company, Incorporated, 3rd edition, 2010.

\bibitem{cummings2023computing}
Joseph Cummings and Benjamin Hollering.
\newblock Computing implicitizations of multi-graded polynomial maps, 2023.

\bibitem{cummings2024invariants}
Joseph Cummings, Benjamin Hollering, and Christopher Manon.
\newblock Invariants for level-1 phylogenetic networks under the {C}avendar-{F}arris-{N}eyman model.
\newblock {\em Advances in Applied Mathematics}, 153:102633, 2024.

\bibitem{evans1993invariants}
Steven~N. Evans and T.~P. Speed.
\newblock Invariants of some probability models used in phylogenetic inference.
\newblock {\em Ann. Statist.}, 21(1):355--377, 1993.

\bibitem{farris1973probability}
James~S Farris.
\newblock A probability model for inferring evolutionary trees.
\newblock {\em Systematic Zoology}, 22(3):250--256, 1973.

\bibitem{Gorla2007}
Elisa Gorla.
\newblock Mixed ladder determinantal varieties from two-sided ladders.
\newblock {\em J. Pure Appl. Algebra}, 211(2):433--444, 2007.

\bibitem{M2}
Daniel~R. Grayson and Michael~E. Stillman.
\newblock Macaulay2, {V}ersion 1.20, 2022.
\newblock {\tt http://www.math.uiuc.edu/Macaulay2/}.

\bibitem{gross2023dimensions}
Elizabeth Gross, Robert Krone, and Samuel Martin.
\newblock Dimensions of level-1 group-based phylogenetic networks.
\newblock {\em Bulletin of Mathematical Biology}, 86:90, 2024.

\bibitem{gross2018distinguishing}
Elizabeth Gross and Colby Long.
\newblock Distinguishing phylogenetic networks.
\newblock {\em SIAM Journal on Applied Algebra and Geometry}, 2(1):72--93, 2018.

\bibitem{gross2020phylogenetic}
Elizabeth Gross, Colby Long, and Joseph Rusinko.
\newblock Phylogenetic networks.
\newblock {\em A Project-Based Guide to Undergraduate Research in Mathematics: Starting and Sustaining Accessible Undergraduate Research}, pages 29--61, 2020.

\bibitem{gross2021distinguishing}
Elizabeth Gross, Leo van Iersel, Remie Janssen, Mark Jones, Colby Long, and Yukihiro Murakami.
\newblock Distinguishing level-1 phylogenetic networks on the basis of data generated by markov processes.
\newblock {\em Journal of Mathematical Biology}, 83:1--24, 2021.

\bibitem{hendy1996complete}
Michael~D Hendy and David Penny.
\newblock Complete families of linear invariants for some stochastic models of sequence evolution, with and without the molecular clock assumption.
\newblock {\em Journal of Computational Biology}, 3(1):19--31, 1996.

\bibitem{hollering2021identifiability}
Benjamin Hollering and Seth Sullivant.
\newblock Identifiability in phylogenetics using algebraic matroids.
\newblock {\em Journal of Symbolic Computation}, 104:142--158, 2021.

\bibitem{kimura1981estimation}
Motoo Kimura.
\newblock Estimation of evolutionary distances between homologousnucleotide sequences.
\newblock {\em Proceedings of the National Academy of Sciences}, 78(1):454--458, 1981.

\bibitem{ColbyPfaff}
Colby Long.
\newblock {Initial ideals of Pfaffian ideals}.
\newblock {\em Journal of Commutative Algebra}, 12(1):91 -- 105, 2020.

\bibitem{martin2023algebraic}
Samuel Martin, Vincent Moulton, and Richard~M. Leggett.
\newblock Algebraic invariants for inferring 4-leaf semi-directed phylogenetic networks.
\newblock {\em bioRxiv 2023.09.11.557152}, 2023.

\bibitem{mpmath}
The mpmath~development team.
\newblock {\em mpmath: a {P}ython library for arbitrary-precision floating-point arithmetic}, 2023.
\newblock {\tt http://mpmath.org/}.

\bibitem{mathrepo}
{MATHREPO} {Mathematical Data and Software}.
\newblock \url{https://mathrepo.mis.mpg.de/PfaffianPhylogeneticNetworks}, 2023.
\newblock [Online; accessed 17 October 2023].

\bibitem{mumfordRedBook}
David Mumford.
\newblock {\em The red book of varieties and schemes}, volume 1358 of {\em Lecture Notes in Mathematics}.
\newblock Springer-Verlag, Berlin, expanded edition, 1999.
\newblock Includes the Michigan lectures (1974) on curves and their Jacobians, With contributions by Enrico Arbarello.

\bibitem{neyman1971statistical}
Jerzy Neyman.
\newblock Molecular studies of evolution: A source of novel statistical problems.
\newblock In Shanti~S. Gupta and James Yackel, editors, {\em Statistical Decision Theory and Related Topics}, pages 1--27. Elsevier, 1971.

\bibitem{norris1997markov}
J~R Norris.
\newblock {\em Markov Chains}.
\newblock Cambridge University Press, 1st edition, 1997.

\bibitem{rhodes2012identifiability}
John~A Rhodes and Seth Sullivant.
\newblock Identifiability of large phylogenetic mixture models.
\newblock {\em Bulletin of Mathematical Biology}, 74:212--231, 2012.

\bibitem{scikitbio}
The scikit-bio~development team.
\newblock {\em scikit-bio: A Bioinformatics Library for Data Scientists, Students, and Developers}, 2020.
\newblock {\tt http://scikit-bio.org/}.

\bibitem{Speyer-Sturmfels}
David Speyer and Bernd Sturmfels.
\newblock The tropical {G}rassmannian.
\newblock {\em Adv. Geom.}, 4(3):389--411, 2004.

\bibitem{steel2016phylogeny}
Mike Steel.
\newblock {\em Phylogeny: discrete and random processes in evolution}.
\newblock SIAM, 2016.

\bibitem{sturmfels2005toric}
Bernd Sturmfels and Seth Sullivant.
\newblock Toric ideals of phylogenetic invariants.
\newblock {\em Journal of Computational Biology}, 12(2):204--228, 2005.

\bibitem{Sturmfels2020}
Bernd Sturmfels, Caroline Uhler, and Piotr Zwiernik.
\newblock Brownian motion tree models are toric.
\newblock {\em Kybernetika}, 56(6):1154--1175, 2020.

\bibitem{sullivant2007toric}
Seth Sullivant.
\newblock Toric fiber products.
\newblock {\em J. Algebra}, 316(2):560--577, 2007.

\bibitem{algstat}
Seth Sullivant.
\newblock {\em Algebraic statistics}, volume 194 of {\em Graduate Studies in Mathematics}.
\newblock American Mathematical Society, Providence, RI, 2018.

\bibitem{wu2022ultrafast}
Zhaoxing Wu and Claudia Solis-Lemus.
\newblock Ultrafast learning of 4-node hybridization cycles in phylogenetic networks using algebraic invariants.
\newblock {\em arXiv preprint arXiv:2211.16647}, 2022.

\end{thebibliography}

\end{document}